\documentclass{article}

\usepackage[accepted]{aistats2024}

\usepackage[utf8]{inputenc} 
\usepackage[T1]{fontenc}    
\usepackage{hyperref}       
\usepackage{url}            
\usepackage{booktabs}       
\usepackage{amsfonts}       
\usepackage{nicefrac}       
\usepackage{microtype}      
\usepackage{xcolor}         
\usepackage{hyperref}
\hypersetup{colorlinks=true,linkcolor=blue,filecolor=blue,urlcolor=blue,citecolor=blue}
\usepackage{algorithm}
\usepackage{algorithmic}
\usepackage{graphicx}
\usepackage{amsmath}
\usepackage{amssymb}
\usepackage{mathtools}
\usepackage{amsthm}
\usepackage{mathrsfs}
\usepackage{amsmath}
\usepackage{multirow}
\usepackage{multicol}
\usepackage{makecell}
\usepackage{soul}
\usepackage{enumerate}
\usepackage[shortlabels]{enumitem}

\usepackage[round]{natbib}

\usepackage[capitalize,noabbrev]{cleveref}

\theoremstyle{plain}
\newtheorem{theorem}{Theorem}[section]
\newtheorem{proposition}[theorem]{Proposition}
\newtheorem{lemma}[theorem]{Lemma}

\theoremstyle{definition}
\newtheorem{definition}[theorem]{Definition}

\theoremstyle{remark}
\newtheorem{remark}[theorem]{Remark}
\DeclareMathOperator{\argmin}{argmin}
\def\x{{\mathbf{x}}}
\def\b{{\mathbf{b}}}
\def\y{{\mathbf{y}}}
\def\z{{\mathbf{z}}}

\def\w{{\mathbf{w}}}
\def\u{{\mathbf{u}}}

\def\r{{\mathbf{r}}}
\def\v{{\mathbf{v}}}
\def\a{{\mathbf{a}}}
\def\b{{\mathbf{b}}}
\def\bz{{\bf 0}}

\def\cX{{\mathcal{X}}}
\def\cD{{\mathcal{D}}}

\def\N{{\mathbb{N}}}

\def\lbd{{\lambda}}
\newcommand{\R}{\mathbb{R}}

\newcommand{\blam}{{\boldsymbol\lambda}}
\newcommand{\bga}{{\boldsymbol\gamma}}

\newcommand{\bro}{{\boldsymbol{\rho}}}
\newcommand{\bth}{{\boldsymbol{\theta}}}
\newcommand{\bmu}{{\boldsymbol{\mu}}}
\newcommand{\bbe}{{\boldsymbol{\beta}}}
\newcommand{\dist}{\operatorname{dist}}

\begin{document}
\twocolumn[

\aistatstitle{Lower-level Duality Based Reformulation and Majorization Minimization Algorithm for Hyperparameter Optimization} 

\aistatsauthor{ He Chen$^{1}$ \And Haochen Xu$^{2}$ \And Rujun Jiang$^{2\,\dag
}$ \And Anthony Man-Cho So$^{1}$ }

\aistatsaddress{ $^{1}$Dept.\ of SEEM,
 The Chinese University of Hong Kong\\ \ $^{2}$School of Data Science, Fudan University }

]

\begin{abstract}
Hyperparameter tuning is an important task of machine learning, which can be formulated as a bilevel program (BLP).
 However, most existing algorithms are not applicable for BLP with non-smooth lower-level problems.
To address this, we propose a single-level reformulation of the BLP based on lower-level duality without involving any implicit value function.
To solve the reformulation, we propose a majorization minimization algorithm that marjorizes the constraint in each iteration.
Furthermore, we show that the subproblems of the proposed algorithm for several widely-used hyperparameter turning models can be reformulated into conic programs that can be efficiently solved by the off-the-shelf solvers.
We theoretically prove the convergence of the proposed algorithm and demonstrate its superiority through numerical experiments.
\end{abstract}
\bibliographystyle{apalike}
\section{Introduction}
\label{Introduction}
Machine learning research is focused on developing methods that can effectively extract important elements from given datasets. Various learning methods has emerged, encompassing biologically inspired neural networks \citep{bishop1995neural}, ensemble models \citep{claesen2014ensembleSVM}, adversarial learning \citep{bruckner2011stackelberg,wang21d,Wang2022SolvingSP}, and reinforcement learning \citep{yang2019provably,wu2020finite}.
These methods commonly rely on a set of hyperparameters, which are adjustable parameters that configure various aspects of the learning algorithm. The choice of hyperparameters can significantly impact the resulting model and its performance, leading to a wide range of effects.

Finding the optimal hyperparameters for a machine learning model is often considered one of the most challenging aspects of the workflow. Regularization, a widely employed technique in model fitting for regression and classification tasks, involves adding a regularization penalty to the empirical risk term, thereby controlling complexity. An advanced strategy for adapting hyperparameters is to employ a training/validation approach, which entails optimizing the parameters with regularization on a training set and subsequently evaluating the performance by computing its loss on a separate validation set. Mathematically, the process of hyperparameter selection can be formulated into the following bilevel program (BLP):
\begin{equation}\label{eq1}
\begin{array}{rl}
     \min\limits_{\x\in\R^n,\blam\in\R_+^\tau}& L(\x) \\
     {\rm s.t.\quad}&~ \x\in\mathop{\arg\min}\limits_{\hat{\x}\in\R^n}\left\{l(\hat{\x})+\sum\limits_{i=1}^{\tau}\lambda_iP_i(\hat{\x})\right\},
\end{array}
\end{equation}
where $L,l,P_i:\R^n\rightarrow \R\cup\{+\infty\}$ are  proper convex closed  functions, $\x$ is the parameter to learn, and $\blam$ is a vector of hyperparameters. Note that all the functions can be \emph{nonsmooth}. In BLP \eqref{eq1}, the upper-level (UL) problem minimizes the validation error affected by the hyperparameters, and the lower-level (LL) problem aims to minimize structural risk on given training data incorporating a regularizer penalized by hyperparameters that need to be tuned. Table \cref{table1} provides some illustrative examples of bilevel hyperparameter selection problems in form \eqref{eq1}.
\subsection{Related Work}
In the existing literature, various approaches have been proposed for hyperparameter selection. The simpler approaches include brute force grid search and Bayesian optimization, which handle hyperparameters and datasets of small-scale but suffer from high computational requirements. These gradient-free methods face limitations when dealing with a large number of parameters.

Gradient-based methods for BLPs are popular in literature. It can be broadly categorized into two groups. Explicit Gradient-Based Methods (EGBMs) \citep{franceschi2017forward,franceschi2018bilevel} utilize dynamic frameworks and iterative algorithms to solve the LL problem. 
Implicit Gradient-Based Methods (IGBMs) \citep{pedregosa2016hyperparameter,rajeswaran2019meta,lorraine2020optimizing} rely on the first-order optimality condition for the LL problem and the chain rule to derive the hyper-gradient by solving a linear system. To mitigate computational complexity, techniques such as the Conjugate Gradient (CG) method or the Neumann method \citep{pedregosa2016hyperparameter,lorraine2020optimizing} are often employed for fast inverse computations. Recently, \cite{liubome} introduce an approach that drops the implicit gradient and \cite{shen2023penalty} tackle the bilevel problem through the penalty methods. However, all the above mentioned methods are only applicable to smooth LL problems, and may not be suitable for \eqref{eq1}.

For nonsmooth functions, \cite{bertrand2020implicit} propose a new implicit differentiation method combined with block coordinate descent to solve Lasso-type models for hyperparameter optimization. In their subsequent work \citep{bertrand2022implicit}, it is extended to solve more general non-smooth hyperparameter optimization problems. However, their methods are restricted to $l_1$ regularized LL problems, which cannot deal with general $P_i$. In the context of difference of convex bilevel programs, \cite{ye2021difference} develop a numerical algorithm called iP-DCA and applies it to hyperparameter selection, particularly in support vector machine models. \cite{gao2022value}  propose the Value Function Based Difference-of-Convex Algorithm (VF-iDCA) to handle BLPs like the one presented in equation (1), where the LL problems involve complex regularization terms. However, these DCA-based methods requires to compute the optimal value of the LL problem to obtain a subgradient, which is used to linearizes the concave term in DC constraint at each iteration.  Recently, \cite{chen2023bilevel} have introduced an inexact gradient-free method whose subproblem is a simple bilevel program, which is still difficult to solve.

\begin{table*}
  \caption{{Examples of bilevel hyperparameter selection problems of the form \eqref{eq1}, see \cite{kunapuli2008classification,feng2018gradient} for reference.}}
 \label{table1}
 \centering
 \small
 \setlength\tabcolsep{5pt}
 \resizebox{\textwidth}{12mm}{
 \begin{tabular}{lcccc}
 \toprule
 Machine learning algorithm & LL variable & UL variable & $L(\x)/l(\x)$ & Regularization\\
 \midrule
 elastic net & $\x$ & $\lambda_1,\lambda_2$ & $\frac{1}{2}\sum_{i\in I_{val}/i\in I_{tr}}|b_i-\x^T\mathbf{a}_i|^2$ & $\lambda_1\|\x\|_1+\frac{\lambda_2}{2}\|\x\|_2^2$ \\
 sparse group lasso & $\x$ & $\lambda\in\R_+^{M+1}$ & $\frac{1}{2}\sum_{i\in I_{val}/i\in I_{tr}}|b_i-\x^T\mathbf{a}_i|^2$ &  $\sum_{m=1}^M\lambda_m\|\x^{(m)}\|_2+\lambda_{M+1}\|\x\|_1$\\
  support vector machine & $\mathbf{w},c$ & $\lambda,\overline{\mathbf{w}}$ & $\sum_{j\in I_{val}/j\in I_{tr}} max(1-b_j(\mathbf{x}^T\mathbf{a}_j-c),0)$ & $\frac{\lambda}{2}\|\mathbf{x}\|^2$\ (with constraint $-\overline{\mathbf{w}}\leq\x\leq\overline{\mathbf{w}}$)\\
 low-rank matrix completion & $\mathbf{\theta},\mathbf{\beta},\Gamma$ & $\lambda\in\R_+^{2G+1}$ &  $\sum_{(i,j)\in\Omega_{val}/(i,j)\in\Omega_{tr}}\frac{1}{2}|M_{ij}-\x_i\mathbf{\theta}-\mathbf{z}_j\mathbf{\beta}-\Gamma_{ij}|$ & $\lambda_0\|\Gamma\|_*+\sum_{g=1}^G\lambda_g\|\mathbf{\theta}^{(g)}\|_2+\sum_{g=1}^G\lambda_{g+G}\|\mathbf{\beta}^{(g)}\|_2$\\
 \bottomrule
 \end{tabular}
 }\vskip -0.15in
 \end{table*}

\subsection{Our Motivations and Contributions}
This paper presents a novel single-level reformulation for the structured BLP \eqref{eq1}. 
 By leveraging Fenchel's duality, the proposed reformulation only requires the expression of the conjugate of each atom function, which is often more easily accessible compared to the value function used in previous works \citep{bertrand2022implicit,gao2022value,pedregosa2016hyperparameter}.
Moreover, we introduce a Lower-level Duality based Majorization Minimization Algorithm (LDMMA) for hyperparameter selection in form \eqref{eq1}. Notably, this algorithm accommodates lower-level problems that are \emph{nonsmooth} and \emph{non-strongly convex} in $\mathbf{x}$.
We first reformulate \eqref{eq1} into a single-level problem without involving any value function.
However, a drawback of this convex subproblem is that it lacks an interior point, rendering most convex optimization methods ineffective. To remedy this, a small positive constant $\epsilon$ is added to the right side of the constraint in \eqref{eq1}, leading to a relaxed convex approximation with interior points.
Based on this reformulation, we propose an iterative algorithm that sequentially solves convex subproblems using majorization minimization (MM) techniques to approximate the constraint.
Furthermore, we show that the subproblems of several widely-used hyperparameter models can be reformulated to conic programs that can be efficiently solved by the off-the-shelf solvers.
Additionally, we demonstrate that the obtained solutions converge sub-sequentially to a Karush-Kuhn-Tucker (KKT) point of the $\epsilon$-perturbation of problem \eqref{eq1} under mild conditions.
Numerical experiments are conducted to showcase the efficiency of our method. 
We summarize our contributions as follows
\begin{itemize}
    \item We provide a novel reformulation for a class of structured BLPs that is a single-level problem and do not involve value functions.
     \item Based on the reformulation and using MM techniques, we further propose an iterative algorithm, LDMMA, where the subproblem in each iteration is a convex problem. For many practical applications, the subproblem is a convex conic program.
      \item Theoretically, we prove that our algorithm generates a sequence whose accumulation points are KKT points under mild conditions.
    \item We conduct numerical experiments on both synthetic and real-world datasets and show that LDMMA exceeds the state-of-the-art.
\end{itemize}

\section{Fenchel's Duality Based Reformulation for BLP}\label{sec:fenchel}
We propose a Fenchel's duality based reformulation for the following problem, which is a generalization of \eqref{eq1}
\begin{equation}\label{eq:abstract_obj}
     \min\limits_{\x\in\R^n,\blam\in\cD} f(\x,\blam) \qquad
     {\rm s.t.}~ \x\in\mathop{\arg\min}\limits_{\hat{\x}}\sum \limits_{i=0}^{\tau}g_i(\hat{\x},\blam),
\end{equation}
where $\blam$ is a vector of hyperparameters in a convex closed set $\cD$ and $f,g_i$ are proper convex closed function in $\x$ and $\blam$ but possibly non-smooth. The core idea is to replace the $\min$ operator in LL problem with $\max$ operator by invoking Fenchel's duality in the conventional value function reformulation, and then the $\max$ operator can be omitted due to the direction of the inequality. Hence we obtain an equivalent inequality constraint only involving LL functions and their conjugates. Let us begin with the following equivalent form of LL problem.
\begin{equation}\label{eq:LL1}
    \min\limits_{\x}g_0(\x,\blam)+\sum\limits_{i=1}^{\tau} g_i(\z_i,\blam)\qquad {\rm s.t.}\ \ \x=\z_i.
\end{equation}
Since $g_i,i=0,1,\ldots,\tau$ are convex and the constraints are affine, it is know that strong duality holds under Slater's condition. That is, if $\cap_{i=0}^{\tau}{\rm ri}({\rm dom}\ g_i(\cdot,\blam))\neq\emptyset$\footnote{Here, ${\rm ri}(\cdot)$ denotes the relative interior of the set $\cdot$.}, \eqref{eq:LL1} is equivalent to the following problem:
{\small
\begin{equation}\label{eq:LL2}
    -\min\limits_{\boldsymbol{\rho}_i}\max\limits_{\x,\z_i} -g_0(\x,\blam)-\sum\limits_{i=1}^{\tau}g_i(\z_i,\blam)-\sum\limits_{i=1}^{\tau}\boldsymbol{\rho}_i^T(\x-\z_i).
\end{equation}
}
Here $\boldsymbol{\rho}_i\in\R^n,i=1,\ldots,\tau$ are Lagrangian multipliers associated with constraint $\x=\z_i$, and the $\min$ and $\max$ operators have been exchanged by adding the negative signs. We define $g^*_i(\y,\blam):=\max_{\x}\ \y^T\x-g_i(\x,\blam)$ as the conjugate functions regarding $\x$ for $g_i$. We then simplify \eqref{eq:LL2} as
\begin{equation*}
    \max\limits_{\boldsymbol{\rho}_i} -g_0^*\left(-\sum\limits_{i=1}^{\tau}\boldsymbol{\rho}_i,\blam\right)-\sum\limits_{i=1}^{\tau}g^*_i(\boldsymbol{\rho}_i,\blam).
\end{equation*}
Note that the constraint of problem \eqref{eq:abstract_obj}, namely, 
$\mathbf{x}\in\arg\min_{\hat{\mathbf{x}}} \sum_{i=0}^{\tau}g_i(\hat{\mathbf{x}},\boldsymbol{\lambda})$, is equivalent to 
\[
\begin{aligned}
    \sum _{i=0}^{\tau}g _i(\mathbf{x},\boldsymbol{\lambda})&\leq \min _{\mathbf{x}}\sum _{i=0}^{\tau}g _i(\mathbf{x},\boldsymbol{\lambda})\\
    &=\max _{\boldsymbol{\rho} _i}-g _0^*\left(-\sum _{i=0}^{\tau}\boldsymbol{\rho} _i,\boldsymbol{\lambda}\right)-\sum _{i=1}^{\tau}g _i^*(\boldsymbol{\rho} _i,\boldsymbol{\lambda}).
\end{aligned}
\]
We can remove the $\max$ operator and find the identical constraint that
\begin{equation*}
     \sum\limits_{i=0}^{\tau} g_i(\x,\blam)+g_0^*\left(-\sum\limits_{i=1}^{\tau}\boldsymbol{\rho}_i,\blam\right)+\sum\limits_{i=1}^{\tau}g_i^*(\boldsymbol{\rho}_i,\blam)\leq0.
\end{equation*}
The result is summarized in the following theorem.
\begin{theorem}\label{th:fenchel}
    Given convex, lower semi-continuous functions $f$ and $g_i$, if $\bigcap_{i=0}^{\tau}{\rm ri}({\rm dom}\ g_i)\neq\emptyset$, then Problem \eqref{eq:abstract_obj} has the following equivalent form:
    { 
    \begin{equation}\label{eq:fenchel}
    \begin{array}{rl}
          \min\limits_{\x,\boldsymbol{\rho}_i\in\R^n,\blam\in\cD}&f(\x,\blam) \\
        {\rm s.t.\quad \ }&\sum\limits_{i=0}^{\tau} g_i(\x,\blam)+g_0^*\left(-\sum\limits_{i=1}^{\tau}\boldsymbol{\rho}_i,\blam\right)\\
        & +\sum\limits_{i=1}^{\tau}g_i^*(\boldsymbol{\rho}_i,\blam)\leq0.
    \end{array}
    \end{equation}
    }
\end{theorem}
The main benefit of reformulation \eqref{eq:fenchel} is circumventing the computation of complex value functions. Instead, it reduces to calculate the conjugate of each atom function $g_i$ respectively, which has closed-form expression in many practical problems. We then demonstrate the power of this reformulation in hyperparameter selection problems.
As a straightforward application of Theorem \ref{th:fenchel}, \eqref{eq1} is equivalent to
{
\begin{equation}\label{eq6}
\begin{aligned}
     \min\limits_{\x,\boldsymbol{\rho}_i\in\R^n,\blam\in\R^{\tau}_+}&  L(\x) \\
        {\rm s.t.\quad\ \ } & F(\x,\blam,\bro)+\sum_{i=1}^{\tau} \lambda_iP_i(\x )\leq 0,
\end{aligned}
\end{equation}
}
  where we use the conventions $0P_i^*(\frac{\bro_i}{0})=0$\footnote{By definition, $\lbd_iP_i^*(\frac{\bro_i}{\lbd_i})=\max_{\z_i}\bro_i^T\z_i-\lambda_i P(\z_i)$. When $\lambda_i=0$, $\lbd_iP_i^*(\frac{\bro_i}{\lbd_i})=\max_{\z_i}\bro_i^T\z_i$ is $0$ if $\bro_i=0$ and $\infty$ otherwise. The latter case contradicts strong duality and thus is abandoned.  } { and }
  {\small
  \begin{equation}\label{eq:f}
F(\x,\blam,\bro)=l(\x)+l^*\left(-\sum\limits_{i=1}^{\tau}{\bro}_i\right)+\sum\limits_{i=1}^{\tau}\lbd_iP_i^*\left(\frac{\bro_i}{\lbd_i}\right).
\end{equation}
    }
By introducing an auxiliary variables $r_i$ satisfying $P_i(\x)\leq r_i$, since $\lbd_i\geq0$, constraint \eqref{eq6} is equivalent to
{ 
\[\begin{array}{l}
     F(\x,\blam,\bro)+\sum\limits_{i=1}^{\tau} \lambda_ir_i\leq 0,\\
    P_i(\x)\leq r_i\text{ for }i\in [\tau]. 
\end{array}\] }
      
This directly gives the following result.
\begin{proposition}\label{pro2.1}
Problem \eqref{eq1} can be reformulated as the following problem.
\begin{equation}\label{eq:ref}
    \begin{array}{lcl}
       &\min\limits_{\x,\bro_i\in\R^n,\mathbf{r},\blam\in\R^{\tau}_+} & L(\x) \\
    &{\rm s.t.} &
    F(\x,\blam,\bro) + \sum_{i=1}^{\tau}  \lambda_ir_i \leq 0,\\
    & &P_i(\x)\leq r_i,~i\in [\tau], \quad \blam\ge0.\\
    \end{array}
\end{equation}
\end{proposition}
Note that the function $F(\x,\blam,\bro)$ is convex because the conjugate functions $l^*$ and $P_i^*$ are convex, and $\lbd_iP_i^*\left(\frac{\bro_i}{\lbd_i}\right)$ is convex as it is the perspective  of function $P_i^*(\bro_i)$ \citep{boyd2004convex}.
We remark that the reformulation in \cref{pro2.1} for \eqref{eq:ref} is proposed for the first time. The advantage of this reformulation is that it is a single level problem and does not involve any implicit function like the value function of the LL problem.
\section{Majorization Minimization Algorithm for Hyperparameter Selection}\label{LDMMA}
In this section, we describe an algorithm that utilizes the reformulation \eqref{eq:ref}.
\subsection{Approximation via Majorization Function}
Note that the only nonconvex term in \eqref{eq:ref} is the bilinear term $\sum_{i=1}^{\tau}  \lambda_ir_i$.
We adopt a majorization and minimization technique to handle this nonconvex term \citep{lange2016mm}.
To this end, we define a majorization function as follows.
\begin{definition}\label{def:maj}
We say $m(\bullet,\bullet;\bar\xi,\bar\zeta)$ is a majorization for the bilinear form $\xi\zeta$ at $(\bar\xi,\bar\zeta)$ if it satisfies
\begin{enumerate}
  \item $m(\xi,\zeta;\bar\xi,\bar\zeta)\ge \xi\zeta$ and $m(\bar\xi,\bar\zeta;\bar\xi,\bar\zeta)= \bar \xi\bar\zeta$;
  \item $m(\xi,\zeta;\bar\xi,\bar\zeta)$ is a continuously differentiable function for $(\xi,\zeta)$, and
  $\frac{\partial m(\xi,\zeta;\bar\xi,\bar\zeta)}{\partial\xi}\mid_{(\xi,\zeta)=(\bar\xi,\bar\zeta)}=\bar\zeta$,  $\frac{\partial m(\xi,\zeta;\bar\xi,\bar\zeta)}{\partial\zeta}\mid_{(\xi,\zeta)=(\bar\xi,\bar\zeta)}=\bar\xi$;
  \item  $\frac{\partial m(\xi,\zeta;\bar\xi,\bar\zeta)}{\partial\xi}$ and $\frac{\partial m(\xi,\zeta;\bar\xi,\bar\zeta)}{\partial\zeta}$ are locally Lipschitz continuous with respect to $(\bar{\xi},\bar{\zeta})$.
\end{enumerate}
\end{definition}
There are various ways to construct such majorizations. For instance,
 when $\bar\xi,\bar\zeta>0$, we can set
 \begin{equation}\label{eq:dca_m1}
     m(\xi,\zeta;\bar\xi,\bar\zeta) = \frac{1}{2}\left(\frac{\bar\xi}{\bar\zeta}\zeta^2+\frac{\bar\zeta}{\bar\xi}\xi^2\right)
 \end{equation} by using the Cauchy inequality.
Another method is to use the identity
\[\xi\zeta=\frac{1}{4}(\xi+\zeta)^2 -\frac{1}{4}(\xi-\zeta)^2,\]
and set 
{\small
\begin{equation}\label{eq:dca_m}
m(\xi,\zeta;\bar\xi,\bar\zeta)
=\frac{1}{4}(\xi+\zeta)^2+\frac{1}{4}(\bar\xi-\bar\zeta)^2- \frac{1}{2}(\bar\xi-\bar\zeta)(\xi-\zeta)
\end{equation}
}
by linearizing the second term in the above identity at $(\bar\xi,\bar\zeta)$.

Let $m$ be a majorization of $\xi\zeta$ according to \cref{def:maj}. We now have the following inner approximation of \eqref{eq:ref} at $(\blam^k,\r^k)$,
\begin{equation}\label{eq:refsub}
    \begin{array}{cl}
       \min\limits_{\x,\blam,\bro,\r} & L(\x) \\
    {\rm s.t.} &F(\x,\blam,\bro) + \sum_{i=1}^{\tau} m(\lambda_i,r_i;\lambda_i^k,r_i^k) \leq 0,\\
    &P_i(\x)\leq r_i,~i\in [\tau],\quad \blam\ge0.
    \end{array}
\end{equation}
Traditional MM algorithms solve the convex problem \eqref{eq:refsub} iteratively.
However, we point out that the above problem does not satisfy general constraint qualifications(CQs) like the Slater condition, which
requires that there exists an interior point in the feasible region. Indeed, according to \cref{pro2.1} and item 2 of \cref{def:maj}, we obtain that $F(\x,\blam,\bro) + \sum_{i=1}^{\tau} m(\lambda_i,r_i;\lambda_i^k,r_i^k)\ge0$ for any feasible solution, and thus there does not exist any interior point.

The absence of CQ not only prevents the use of general interior point methods for efficiently solving \eqref{eq:ref} \citep{wright1999numerical}, but also makes it difficult to show the convergence of solutions by sequentially solving \eqref{eq:ref} to KKT points \citep{andreani2016cone}.
To address this, we add a small positive number $\epsilon$ to the right-hand side of the first constraint in \eqref{eq:ref}, and obtain the following approximation problem
\begin{equation}\label{eq:subeps}
 \begin{array}{rl}
       \min\limits_{\x,\blam,\bro,\r} & L(\x) +\frac{\beta}{2}\left\|\left(\x-\x^k,\blam-\blam^k,\r-\r^k,\bro-\bro^k \right) \right\|^2\\
    {\rm s.t.} &F(\x,\blam,\bro) + \sum_{i=1}^{\tau} m(\lambda_i,r_i;\lambda_i^k,r_i^k) \leq \epsilon,\\
     &P_i(\x)\leq r_i,~i\in [\tau], \quad \blam\geq0.
    \end{array}
\end{equation}
Here we also add a proximal term in the objective function to ensure the convergence of our algorithm.
We summarise our method in  \cref{alg:1}.
We remark that line 1 of \cref{alg:1} helps us find a feasible solution for problem \eqref{eq:subeps}, which guarantees the feasibility of problem \eqref{eq:blpeps}, thanks to \cref{def:maj}.

\begin{algorithm}[htbp]
\caption{Lower-level Dual based Majorization Minimization algorithm (LDMMA)}
\label{alg:1}
\begin{algorithmic}[1]
\REQUIRE initial $\epsilon>0,\beta>0$, and ${\boldsymbol{\lambda}}^0\ge 0$.

\STATE Solve the lower-level subproblem $\min_\x l(\x)+\sum_{i=1}^{\tau}\lambda_i^0 P_i(\x)$ and set $r_i^0=P_i(\x)$, $i=1,2,\ldots,\tau$
\FOR{$k=0,1,\dots,$}
    \STATE Solve problem \eqref{eq:subeps} and obtain an optimal solution $(\x^{k+1},\mathbf{r}^{k+1},{\boldsymbol{\lambda}}^{k+1},\bro^{k+1})$
    \IF{Termination criteria is met}
    \STATE Stop
    \ENDIF
\ENDFOR
\end{algorithmic}
\end{algorithm}
\subsection{Conic Formulations of Subproblems}
We point out that for all hyperparameter selection problems in Table \ref{table1},
the subproblems of \eqref{eq:refsub} or \eqref{eq:subeps} have explicit conic convex formulations, which can be solved by existing off-the-shelf solvers efficiently. 

Here, we give an example of the elastic net problem. { Other problems in Table \ref{table1} admit similar conic reformulations. We note that the full row rank condition is not necessary for the conic reformulation. Without such a condition, we can still obtain a conic program for the subproblem but with one extra linear constraint.  See \cref{sec:a} for proofs, remarks, and more details on other problems. }
\begin{proposition}\label{pro:elastic_net}
   Consider the elastic net problem with training data  $A_{tr},\b_{tr}$ and validation data $A_{val},\b_{val}$,
\begin{equation*}\label{eq:elastic}
    \begin{array}{rl}
         \min\limits_{\x} & L(\x) = \frac12\|A_{val}\x-\b_{val}\|_2^2
         \\
         {\rm s.t.} & \x \in{\argmin}\frac12\|A_{tr}\x-\b_{tr}\|_2^2 +\lambda_1\|\x\|_1+\frac{\lambda_2}{2}\|\x\|_2^2.
    \end{array}
\end{equation*}
        If $A_{tr}$ is of full row rank, then using \eqref{eq:dca_m1} or \eqref{eq:dca_m}, we obtain that the subproblem \eqref{eq:refsub}  for the above problem can be reformulated into the following conic program:
     \begin{equation*}
    \begin{array}{lcl}
     &\min\limits_{\x,\blam,\bro,\r,t} & t\\
    &{\rm s.t.}  &  A_{tr}^T\w+ {\bro}_1+\bro_2=\bz,\\
  & & \|\x\|_1\leq r_1,\ \|\bro_1\|_{\infty}\leq\lambda_1,\\
   & &  SOCs(\x,\blam,\bro,\r,t),
    \end{array}
\end{equation*}
 where $SOCs(\x,\blam,\bro,\r,t)$ {  represents second-order cone constraints with variables} $ (\x,\blam,\bro,\r,t)$.
\end{proposition}
{  Equipped with the conic reformulations for the subproblems, one can take clear and concrete steps to apply Algorithm \ref{alg:1} to solve the hyperparameter selection problems. This enhances the implementability of the proposed algorithm, making it practical for real-world applications.

Before we end this section,
we  would like to emphasize the differences between our approach and the duality-based method in \cite{Ouattara2016DualityAT}. First, the main novelty of \cite{Ouattara2016DualityAT} is to use the Lagrangian duality of the LL problem to deal with the constraints of LL problems, while we focus on Fenchel’s duality for unconstrained LL problems. Second, the duality approach in \cite{Ouattara2016DualityAT} still necessitate the calculation of an abstract value function $h$. In contrast, we utilize the splitting structures to obtain a reformulation that only consists of primal atom functions and their conjugates. Our approach circumvents the computation of complex value functions. Third, our reformulation leads to implementable subproblems in the form of conic programs for many problems of interest while that of \cite{Ouattara2016DualityAT} does not. }


\section{Theoretical Investigations}\label{sec:th}
In this section, we show that the sequence generated by \cref{alg:1} { converges to a KKT point of the} $\epsilon$-approximate problem 
\begin{equation}\label{eq:blpeps}
\begin{array}{rl}
       \min\limits_{\x,\blam,\bro,\r}& L(\x)  \\
    {\rm s.t.} & F(\x,\blam,\bro) + \sum_{i=1}^{\tau} \lambda_ir_i \leq \epsilon,\\
    & P_i(\x)\leq r_i,~i\in [\tau], \quad \blam\geq0.
\end{array} 
\end{equation}
Note that in the above problem, we add a positive $\epsilon$ to the LL constraint. We remark that similar techniques are widely used in value function approaches in the literature \cite{liu21o,ye2022difference}. 

We begin with formal definitions of the KKT point and a nonsmooth CQ. Let $N_{\cX}(\x)$ denote the normal cone of the set $\cX$ and $\partial \varphi$ denote the limiting sub-differential of the function $\varphi$ \citep{rockafellar2009variational}.
\begin{definition}
For a constrained optimization
\begin{equation}\label{eq:std}
       \min\limits_{\x\in\cX}  \hat f(\x) \qquad
    {\rm s.t.} ~\hat  h_i(\x)\leq0,\quad i=1,2,\ldots,m,
\end{equation}
we say that $\x^*$ is its KKT point if there exists $\bmu^*\in\R^+$ such that $\mu_i^*\hat h_i(\x^*)=0$, $\hat h_i(\x^*)\leq0$ and
\[\bz\in \partial\hat  f(\x^*)+\sum_{i=1}^m\mu_i^*\partial\hat  h_i(\x^*)+N_{\cX}(\x^*).\]
\end{definition}
The following CQ is the nonsmooth version of the MFCQ that is frequently used for many algorithms.
\begin{definition}[\cite{jourani1994constraint,ye2022difference}]
    Let $\x^*$ be a feasible point of \eqref{eq:std}. We say that the nonzero
abnormal multiplier constraint qualification (NNAMCQ) holds at $\x^*$ for problem \eqref{eq:std} if $\hat h_i(\x^*)<0$ for $i\in[m]$ or $\bz\notin$
{\small
\[ \left\{\sum_{i=1}^m\mu_i\partial \hat h_i(\x^*)+N_{\cX}(\x^*):\mu_i\hat h_i(\x^*)=0,\mu_i\geq0,\bmu\neq\bz\right\}.\]
}
\end{definition}
\begin{lemma}[NNAMCQ]\label{le:cq} {\rm (i)}  Let $\x^*$ be a solution of \eqref{eq:std}. If NNAMCQ holds at $\x^*$, then $\x^*$ is a KKT point of problem \eqref{eq:std}.
  {\rm (ii)} NNAMCQ holds at any feasible point for problem \eqref{eq:blpeps}.
\end{lemma}
Let $\z^k:=(\x^k,\blam^k,\r^k,\bro^k)$ be the $k$-th iteration point of Algorithm \ref{alg:1}. We use the following notations for the concerned problem  \eqref{eq:blpeps} and its subproblem \eqref{eq:subeps}:
    {
    \[
    \begin{array}{rl}
         \cX&=\{(\x,\blam,\r,\bro):\bro=(\bro_1,\ldots,\bro_{\tau}), \blam,\r\geq\bz\},  \\
          \z&=(\x,\blam,\r,\bro)\in\cX,\\
          f(\z)&=L(\x),\\
          f^k(\z) &=L(\x)+\frac{\beta}{2}(\|\r-\r^k\|^2+\|\blam-\blam^k\|^2), \\
         g (\z)&= F(\x,\blam,\bro) + \sum_{i=1}^{\tau}\lambda_ir_i-\epsilon,\\
         \bar g^k (\z)&= F(\x,\blam,\bro) + \sum_{i=1}^{\tau} m(\lambda_i,r_i;\lambda_i^k,r_i^k)-\epsilon,\\
         h_j(\z)&=P_i(\x)-r_i,\ i=1,2,\ldots,\tau.\\
    \end{array}
    \]}
   By the definition of $m$ in \cref{def:maj}, the following lemma naturally holds.
    \begin{lemma}\label{le:direct}For $k=0,1,2,\dots$,
we have the following results: {\rm (i)} $\bar g^k(\z)\geq g (\z)$ and $\bar g^k (\z^k)= g (\z^k)$; {\rm (ii)} $\partial \bar g^k (\z^k)=\partial g (\z^k)$.
    \end{lemma}

    We then introduce a sufficient decrease property of Algorithm  \ref{alg:1}.
\begin{lemma}\label{le:suff}
    Assume $L(\x)$ is bounded below. Then for all $k\in\N$, we have {\rm (i)} $L(\x^{k+1})-L(\x^{k})\leq -\frac{\beta}{2}\|\z^{k+1}-\z^k\|^2$; { \rm (ii)}  $\lim_{k\rightarrow\infty}\|\z^{k+1}-\z^k\|=0.$
\end{lemma}
\begin{theorem}\label{th:main}
Assume $L(\x)$ is bounded below and $\{\z^k\}_{k\in\N}$ is bounded. The following two statements hold:
\begin{enumerate}[label={{\rm (\roman*)}}]
    \item 
    {  If $\epsilon>0$ in \eqref{eq:blpeps}, then any accumulation point of $\{\z^k\}_{k\in\N}$ is a KKT point of \eqref{eq:blpeps};}
\item  Furthermore, if $L(\x),l(\x)$, and $P_i(\x)$, $i\in[\tau]$ are semi-algebraic functions, then $\{\z^k\}_{k\in\N}$ converges to a KKT point of \eqref{eq:blpeps}.
\end{enumerate}

\end{theorem}
The $\epsilon$ perturbation in \eqref{eq:blpeps} is essential, which can be found in many value function based BLP algorithms \citep{ye2022difference,gao2022value,Xu2014ASA}. We suggest referring to \cite{Xu2014ASA,ye2022difference} for the analysis of this relaxation. We remark that boundedness assumptions on $\{\z^k\}_{k\in\N}$ are widely used in relevant literature; see \cite{ye2022difference} and \cite{gao2022value}. We argue its necessity by referring to Theorem 4.2 in \cite{attouch2013convergence}, a well-known convergence result requiring very mild conditions but needs the boundedness of the iterate sequence.
We also remark that the convergence of our algorithm does not require the lower-level problem to be strongly convex, unlike many existing methods for BLP \citep{feng2018gradient,pedregosa2016hyperparameter}.


\section{Experiments}
In this section, we conduct experiments to compare LDMMA with existing algorithms for hyperparameter optimization on synthetic data and real datasets, respectively.
We briefly introduce our competitors in experiments:





\begin{itemize}
    \item \textbf{Grid Search}: We perform a $10\times10$ uniformly-spaced grid search.
    \item \textbf{Random Search}: We uniformly sample $100$ times for each direction of hyperparameters.
    \item \textbf{Implicit Differentiation}: We implement the IGJO algorithm in \cite{feng2018gradient}.
    \item \textbf{TPE}:  We use the Tree-structured Parzen Estimator approach in \cite{bergstra2013making} which is known as a Bayesian optimization method.
    \item \textbf{VF-iDCA}: We implement the VF-iDCA algorithm in \cite{gao2022value}, which considers the LL value function and applies DC program to approximately solve the BLP.
\end{itemize}

We consider hyperparameter optimization for elastic net, sparse group lasso, and support vector machines \citep{kunapuli2008classification,feng2018gradient,gao2022value}. These three models only use a combination of regularization functions $\|\cdot\|_1,\|\cdot\|_2$ and $\frac12\|\cdot\|_2^2$ that are included by our previous analysis. The elastic net \citep{zou2003regression} is a linear combination of the lasso and ridge penalties and the sparse group lasso \citep{simon2013sparse} combines the group lasso and lasso penalties, which are designed to encourage sparsity and grouping of predictors \citep{feng2018gradient}. The support vector machine is a classical machine learning model that assigns labels to objects \citep{noble2006support} and its related BLP has been intensively studied \citep{kunapuli2008classification,couellan2015bi,jiang2020hyper}. To compare the performance of each method, we calculate validation and test error with obtained LL minimizers from solving subproblems in each experiment. Our competitors are implemented using code from \url{https://github.com/SUSTech-Optimization?tab=repositories}. We use the off-the-shelf solver MOSEK\footnote{\url{https://docs.mosek.com/9.3/toolbox/index.html}} to solve the subproblem \eqref{eq:subeps} at each iteration of LDMMA.
The formulations of the three models and their associated subproblems can be found in \cref{sec:a}.

\begin{table*}
\caption{Elastic net problems on synthetic data, where $|I_{tr}|$, $|I_{val}|$, $|I_{te}|$ and $p$ represent the number of training observations, validation observations, predictors and features, respectively.}
\centering
 \setlength\tabcolsep{2pt}
 \resizebox{\linewidth}{!}{
\scalebox{1}{
\begin{tabular}{clccc|cccc}
\toprule
\textbf{Settings} & \textbf{Methods} & \textbf{Time(s)} & \textbf{Val. Err.} & \textbf{Test Err.}&\textbf{Settings}  & \textbf{Time(s)} & \textbf{Val. Err.} & \textbf{Test Err.}\\
\midrule
 \multirow{6}{*}{\makecell[l]{$|I_{tr}|=100$\\ $|I_{val}|=20$\\ $|I_{te}|=250$\\ $p=250$}} & Grid & $6.14\pm2.34$ & $6.31\pm0.84$ & $6.55\pm0.91$ &\multirow{6}{*}{\makecell[l]{$|I_{tr}|=100$\\ $|I_{val}|=100$\\ $|I_{te}|=250$\\ $p=450$}}& $7.49\pm0.26$ & $7.40\pm1.29$ & $6.05\pm1.06$\\
 & Random & $9.54\pm0.28$ & $5.98\pm2.24$ & $6.56\pm0.89$& &$17.12\pm0.40$ & $7.40\pm1.31$ & $7.10\pm1.08$ \\
 & TPE & $10.10\pm0.44$ & $6.05\pm1.35$ & $6.53\pm0.90$&& $17.86\pm0.92$ & $7.38\pm1.30$ & $7.06\pm1.06$\\
 & IGJO & $3.92\pm2.42$ & $4.46\pm1.75$ & $6.76\pm0.97$&& $4.02\pm2.99$ & $5.63\pm1.24$ & $5.36\pm1.07$\\
 & VF-iDCA & $0.84\pm0.25$ & $2.14\pm0.76$ & $4.03\pm0.65$&&$2.57\pm0.96$ & $3.64\pm0.53$ & $4.73\pm0.69$\\
 & LDMMA & $0.64\pm0.20$ & $2.06\pm0.42$ & $3.91\pm0.63$ \vspace{2pt}&&$1.85\pm0.21$ & $3.15\pm0.32$ & $4.25\pm0.48$\\
 \hline \\ [-1.8ex]
 \multirow{6}{*}{\makecell[l]{$|I_{tr}|=100$\\ $|I_{val}|=100$\\ $|I_{te}|=250$\\ $p=250$}} & Grid & $9.71\pm0.21$ & $6.82\pm1.14$ & $6.55\pm0.91$ & \multirow{6}{*}{\makecell[l]{$|I_{tr}|=100$\\ $|I_{val}|=100$\\ $|I_{te}|=100$\\ $p=2500$}} & $13.17\pm3.43$ & $7.81\pm1.53$ & $8.82\pm0.92$ \\
 & Random & $9.54\pm0.28$ & $6.31\pm0.84$ & $6.68\pm1.13$ && $15.29\pm2.60$ & $6.44\pm1.53$ & $8.67\pm0.94$ \\
 & TPE & $10.10\pm0.44$ & $6.30\pm0.85$ & $6.54\pm1.15$ && $22.42\pm1.30$ & $7.71\pm1.32$ & $8.43\pm0.80$\\
 & IGJO & $3.92\pm2.42$ & $4.36\pm0.96$ & $5.54\pm0.82$ && $31.30\pm6.41$ & $7.78\pm1.12$ & $8.61\pm0.82$\\
 & VF-iDCA & $1.90\pm0.56$ & $3.04\pm1.51$ & $4.52\pm0.62$ && $23.57\pm4.06$ & $1.83\pm0.71$ & $5.13\pm1.02$\\
 & LDMMA & $1.12\pm0.15$ & $2.63\pm0.41$ & $4.05\pm0.97$ && $9.30\pm2.61$ & $2.25\pm1.09$ & $4.19\pm0.76$
 \vspace{2pt}\\
\bottomrule
\end{tabular}}}\label{table-ela}
\end{table*}

\begin{table*}
\caption{Sparse group lasso problems on synthetic data, where $p$ and $M$ represent the number of covariates and covariate groups, respectively, and $n$ represent the data scale described above.}
\label{table-sgl}
\centering
\setlength\tabcolsep{2pt}
 \resizebox{\linewidth}{!}{
\scalebox{1}{
\begin{tabular}{clccc|cccc}
\toprule
\textbf{Settings} & \textbf{Methods} & \textbf{Time(s)} & \textbf{Val. Err.} & \textbf{Test Err.} & \textbf{Settings} & \textbf{Time(s)} & \textbf{Val. Err.} & \textbf{Test Err.}\\
\midrule
 \multirow{5}{*}{\makecell[l]{$n=300$\\ $p=600$\\ $M=30$}} & Grid & $35.68\pm1.85$ & $43.80\pm7.31$ & $45.43\pm7.87$ & \multirow{5}{*}{\makecell[l]{$n=450$\\ $p=900$\\ $M=60$}} & $45.72\pm4.88$ & $39.58\pm5.31$ & $46.66\pm5.33$\\
 & Random & $26.32\pm1.51$ & $36.94\pm7.01$ & $43.54\pm8.87$ && $58.58\pm1.24$ & $43.91\pm4.90$ & $41.08\pm9.05$\\
 & IGJO & $49.00\pm4.11$ & $38.90\pm6.21$ & $41.94\pm6.73$&& $64.90\pm10.63$ & $29.90\pm7.15$ & $48.82\pm6.74$\\
 & VF-iDCA & $8.69\pm1.25$ & $0.04\pm0.01$ & $37.31\pm4.01$ && $25.41\pm1.56$ & $20.19\pm6.04$ & $36.36\pm5.45$\\
 & LDMMA & $6.85\pm0.74$ & $22.94\pm2.56$ & $21.25\pm4.63$ && $20.15\pm2.61$ & $21.04\pm2.99$ & $28.83\pm6.76$
 \vspace{2pt}\\
 \hline \\ [-1.8ex]
 \multirow{5}{*}{\makecell[l]{$n=300$\\ $p=900$\\ $M=60$}} & Grid & $40.84\pm1.04$ & $42.45\pm7.67$ & $44.56\pm7.33$ & \multirow{5}{*}{\makecell[l]{$n=600$\\ $p=1200$\\ $M=150$}} & $74.22\pm8.89$ & $50.52\pm4.14$ & $59.90\pm9.01$\\
 & Random & $66.58\pm1.01$ & $39.27\pm7.32$ & $43.00\pm8.83$ & &$72.15\pm4.49$ & $53.21\pm7.64$ & $57.84\pm14.52$\\
 & IGJO & $60.67\pm5.77$ & $28.32\pm4.93$ & $43.43\pm7.44$ && $80.52\pm5.66$ & $41.70\pm5.37$ & $56.01\pm12.74$\\
 & VF-iDCA & $31.75\pm5.62$ & $17.85\pm3.27$ & $32.65\pm4.83$ && $33.57\pm7.48$ & $25.64\pm6.35$ & $29.55\pm3.88$\\
 & LDMMA & $24.78\pm0.92$ & $24.54\pm3.77$ & $24.91\pm3.58$ && $27.34\pm3.73$ & $20.94\pm3.52$ & $23.74\pm2.01$
 \vspace{2pt}\\
\bottomrule
\end{tabular}}}
\end{table*}

\begin{table*}
\tiny
\caption{Support Vector Machine problems with 3-fold and 6-fold cross-validation on three datasets, where the number of features $p$ and samples $|\Omega|,|\Omega_{test}|$ are displayed together with dataset names. Results on other datasets are presented in Appendix \ref{detail-data}.}
\label{table-SVM-3fold-6fold}
\centering
\small
\setlength\tabcolsep{4pt}
\resizebox{\linewidth}{!}{
\begin{tabular}{ll|ccc|ccc}
\toprule
 \multirow{2}{*}{\textbf{Dataset}} & \multirow{2}{*}{\textbf{Methods}} & \multicolumn{3}{c}{\textbf{3-fold}} & \multicolumn{3}{c}{\textbf{6-fold}} \vspace{2pt}\\
 \cline{3-5}\cline{6-8} \multicolumn{1}{c}{}& \multicolumn{1}{c}{} & \textbf{Times(s)} & \textbf{Val. Err.} & \textbf{Test Err.} & \textbf{Times(s)} & \textbf{Val. Err.} & \textbf{Test Err.}\\
 \midrule
 \multirow{5}{*}{\makecell[l]{diabetes-scale \\ $p=8$ \\ $|\Omega|=384$ \\ $|\Omega_{test}|=384$}} & Grid & $3.17\pm0.08$ & $0.55\pm0.03$ & $0.19\pm0.03$ & $6.22\pm0.21$ & $0.54\pm0.03$ & $0.33\pm0.04$\\
 & Random & $3.47\pm0.14$ & $0.56\pm0.03$ & $0.32\pm0.05$ & $7.18\pm0.30$ & $0.55\pm0.04$ & $0.30\pm0.05$ \\
 & TPE & $10.21\pm6.68$ & $0.55\pm0.04$ & $0.29\pm0.06$ & $76.67\pm36.39$ & $0.54\pm0.03$ & $0.34\pm0.06$ \\
 & VF-iDCA & $0.28\pm0.04$ & $0.48\pm0.03$ & $0.23\pm0.01$ & $0.65\pm0.03$ & $0.43\pm0.03$ & $0.23\pm0.02$ \\
 & LDMMA & $0.22\pm0.03$ & $0.49\pm0.02$ & $0.19\pm0.01$ & $0.55\pm0.10$ & $0.39\pm0.05$ & $0.20\pm0.02$
 \vspace{2pt}\\
 \hline \\ [-1.8ex]
 \multirow{5}{*}{\makecell[l]{breast-cancer-scale \\ $p=14$ \\ $|\Omega|=336$ \\ $|\Omega_{test}|=347$}} & Grid & $3.32\pm0.09$ & $0.08\pm0.01$ & $0.16\pm0.08$ & $6.32\pm0.11$ & $0.08\pm0.01$ & $0.15\pm0.12$\\
 & Random & $3.69\pm0.07$ & $0.09\pm0.01$ & $0.08\pm0.08$ & $7.20\pm0.12$ & $0.09\pm0.02$ & $0.10\pm0.11$ \\
 & TPE & $17.88\pm10.05$ & $0.09\pm0.01$ & $0.10\pm0.11$ & $34.66\pm20.57$ & $0.09\pm0.01$ & $0.18\pm0.13$ \\
 & VF-iDCA & $0.24\pm0.04$ & $0.09\pm0.01$ & $0.04\pm0.01$ & $0.57\pm0.12$ & $0.08\pm0.01$ & $0.03\pm0.01$ \\
 & LDMMA & $0.12\pm0.01$ & $0.08\pm0.01$ & $0.03\pm0.01$ & $0.42\pm0.17$ & $0.08\pm0.01$ & $0.02\pm0.01$
 \vspace{2pt}\\
 \hline \\ [-1.8ex]
 \multirow{5}{*}{\makecell[l]{w1a \\ $p=300$ \\ $|\Omega|=1236$ \\ $|\Omega_{test}|=1241$}} & Grid & $20.08\pm0.33$ & $0.59\pm0.10$ & $0.41\pm0.14$ & $104.47\pm2.99$ & $0.06\pm0.01$ & $0.03\pm0.00$\\
 & Random & $20.30\pm0.18$ & $0.55\pm0.07$ & $0.31\pm0.08$ & $147.88\pm8.64$ & $0.05\pm0.00$ & $0.02\pm0.00$\\
 & TPE & $85.80\pm13.95$ & $0.64\pm0.13$ & $0.45\pm0.11$ & $682.35\pm17.52$ & $0.06\pm0.01$ & $0.03\pm0.00$\\
 & VF-iDCA & $4.32\pm0.23$ & $0.03\pm0.02$ & $0.03\pm0.00$ & $25.37\pm3.10$ & $0.01\pm0.00$ & $0.03\pm0.00$\\
 & LDMMA & $2.19\pm0.24$ & $0.01\pm0.00$ & $0.01\pm0.00$ & $15.25\pm2.90$ & $0.01\pm0.00$ & $0.02\pm0.00$\\
\bottomrule
\end{tabular}
}
\end{table*}
\begin{figure*}[htbp]
    \centering
      \includegraphics[width=0.32\linewidth]{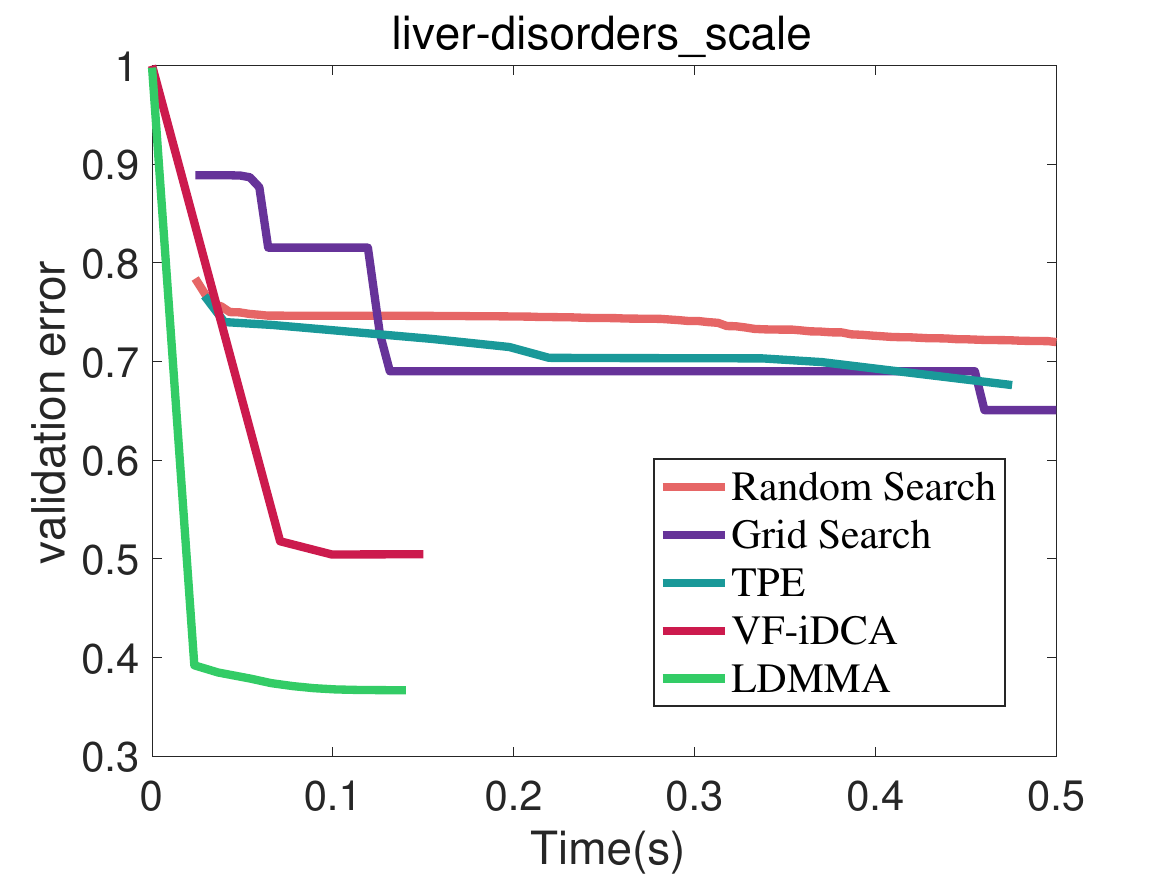}
 		\label{SVM1_val}
 		\includegraphics[width=0.32\linewidth]{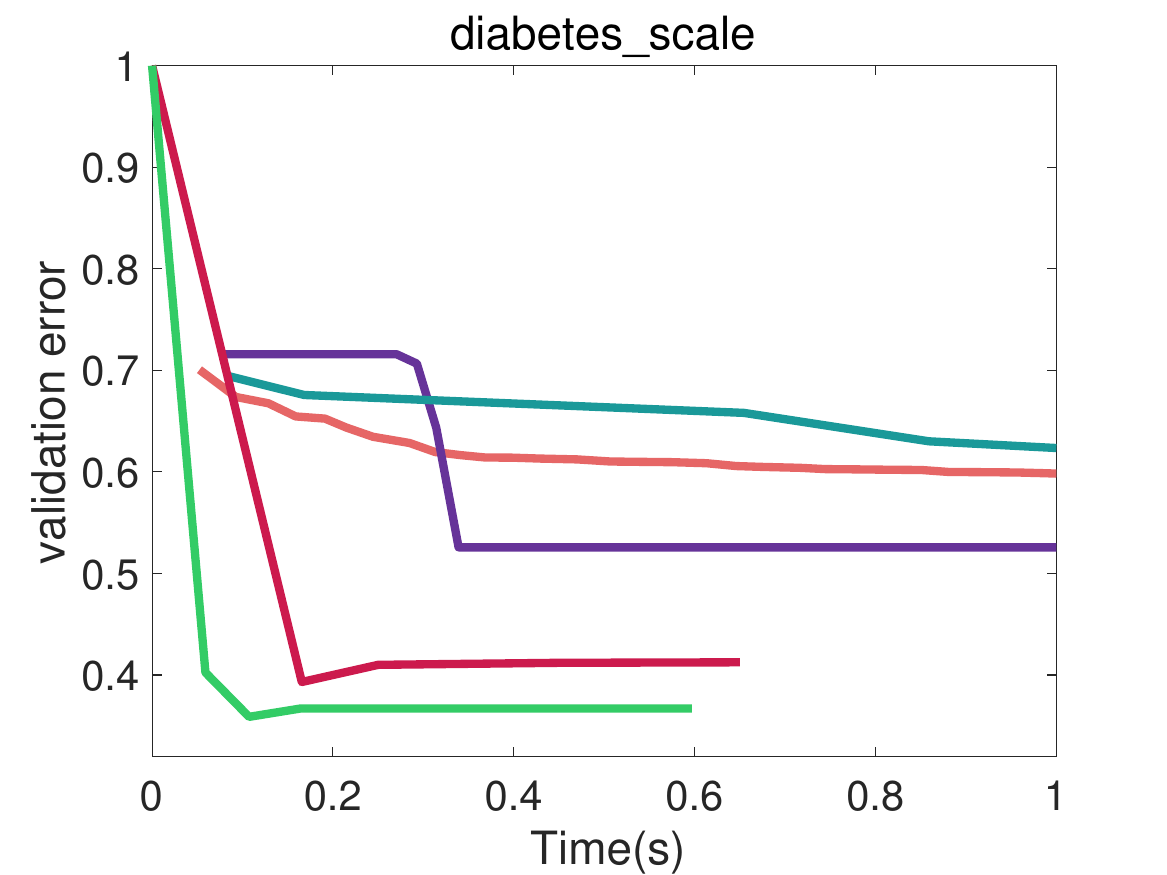}
            \includegraphics[width=0.32\linewidth]{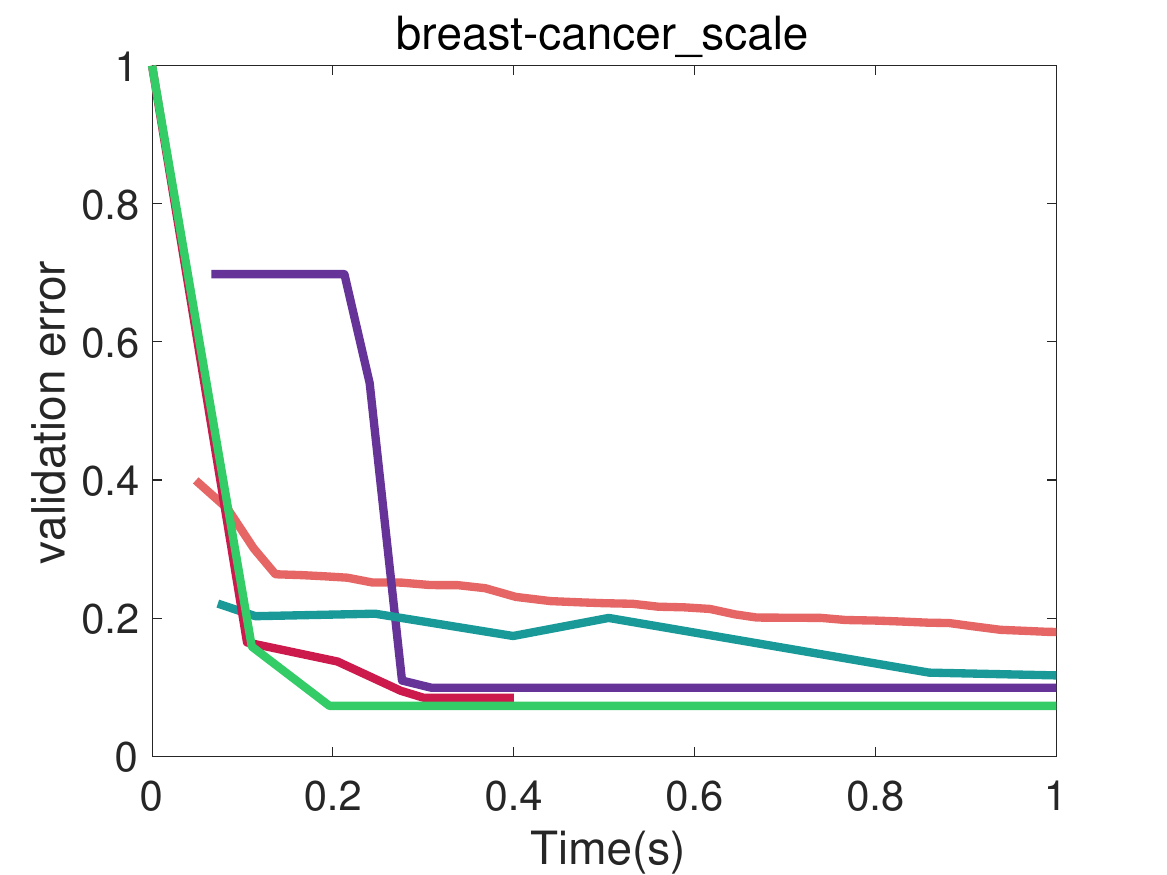}
 		\includegraphics[width=0.32\linewidth]{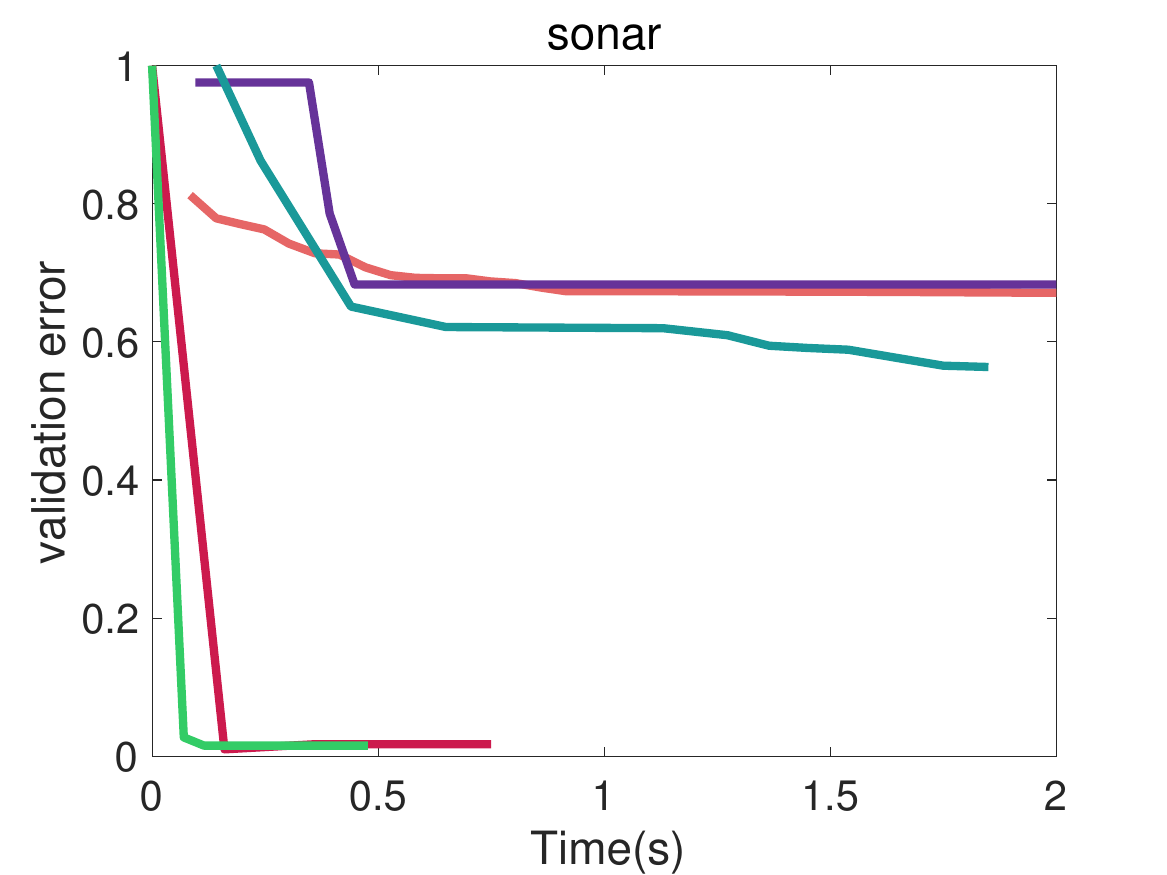}
 		\label{SVM4_val}
            \includegraphics[width=0.32\linewidth]{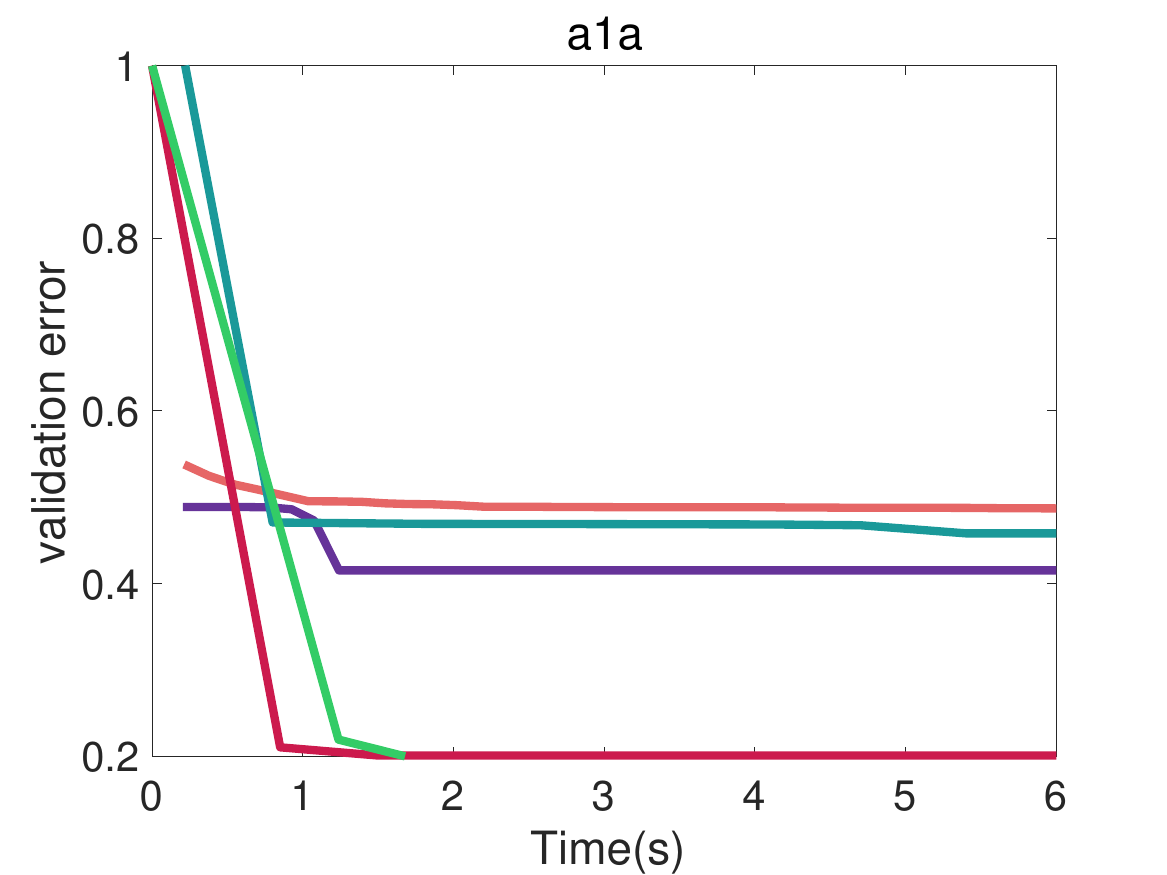}
 		\label{SVM5_val}
 		\centering
 		\includegraphics[width=0.32\linewidth]{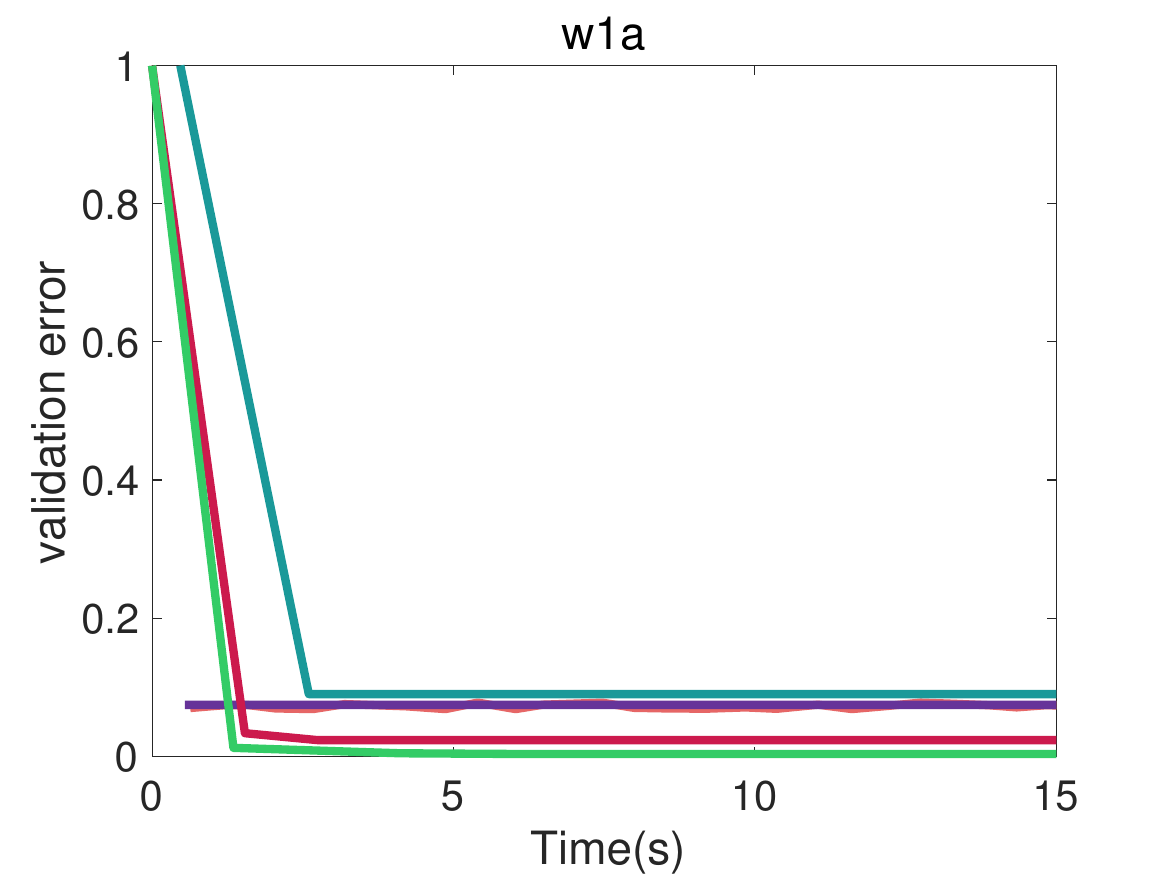}
        \includegraphics[width=0.32\linewidth]{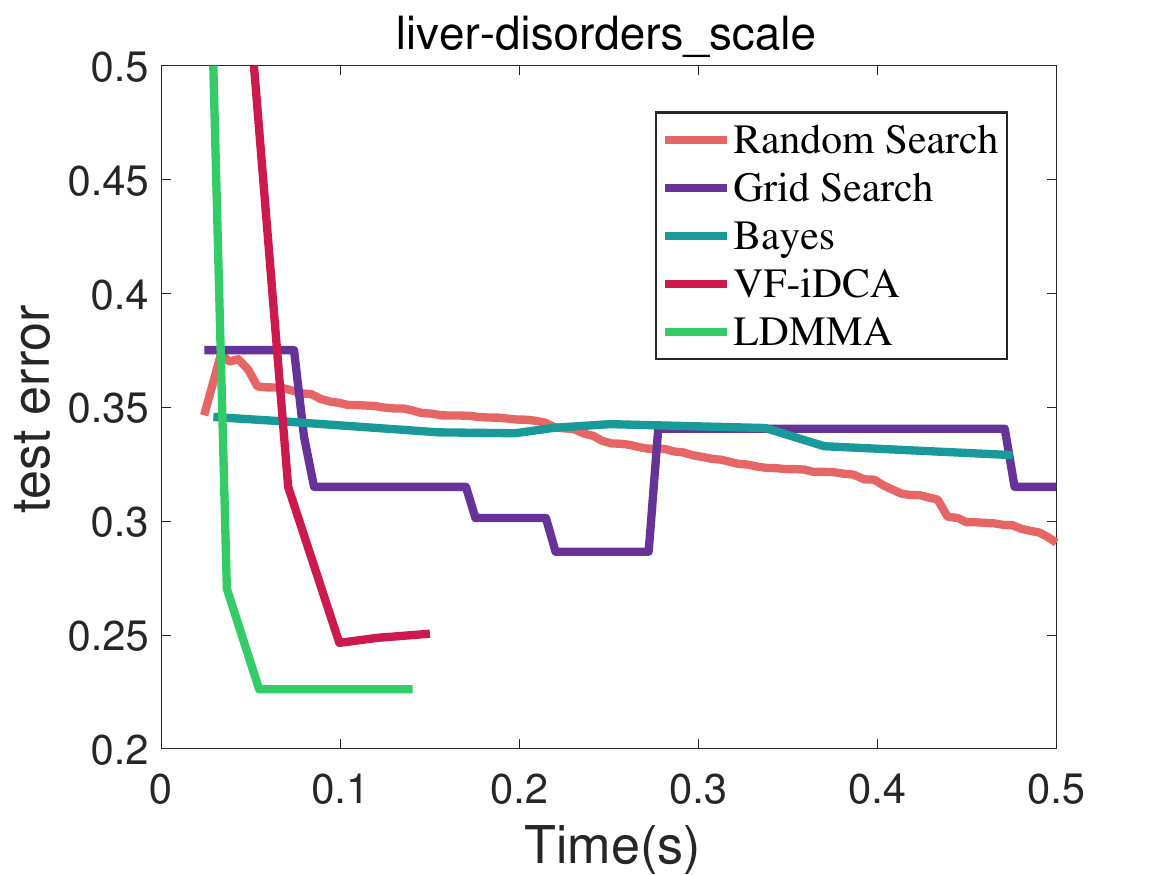}
        \label{SVM1_test}
 		\includegraphics[width=0.32\linewidth]{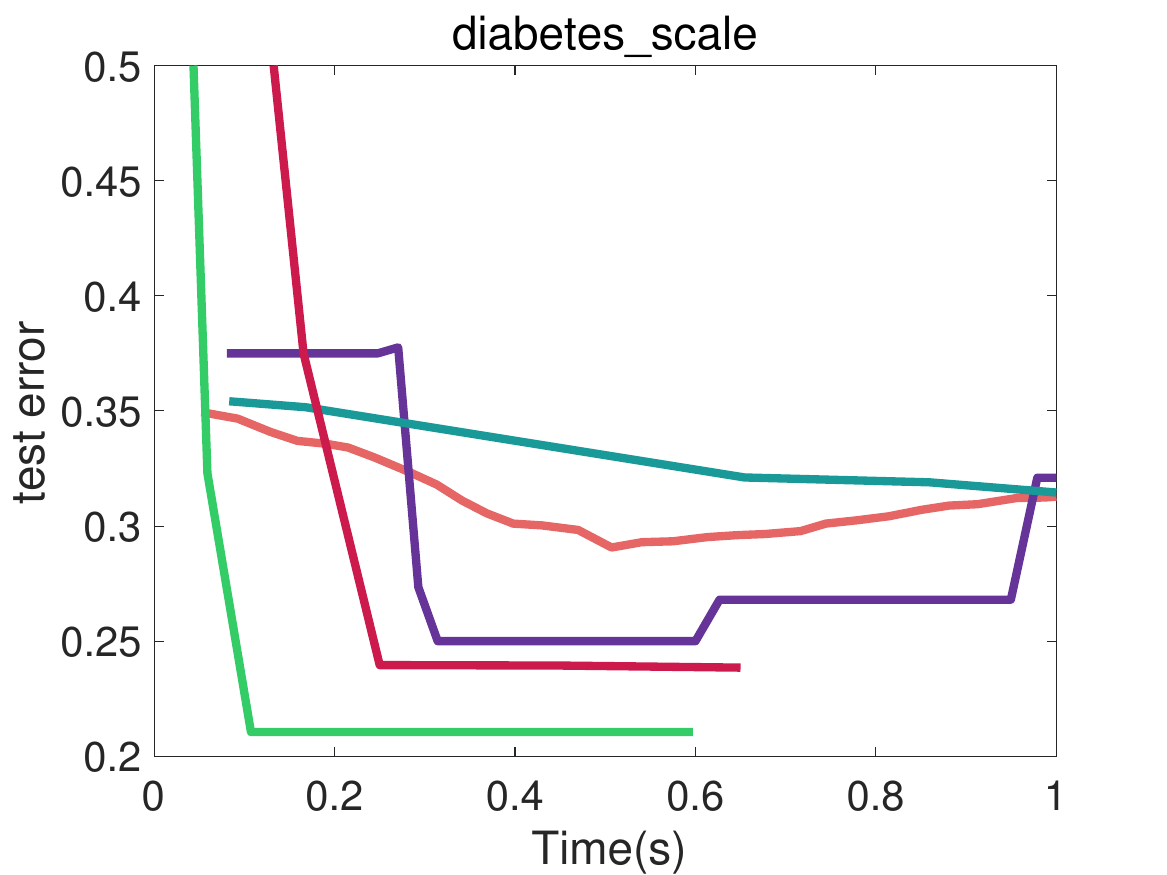}
        \includegraphics[width=0.32\linewidth]{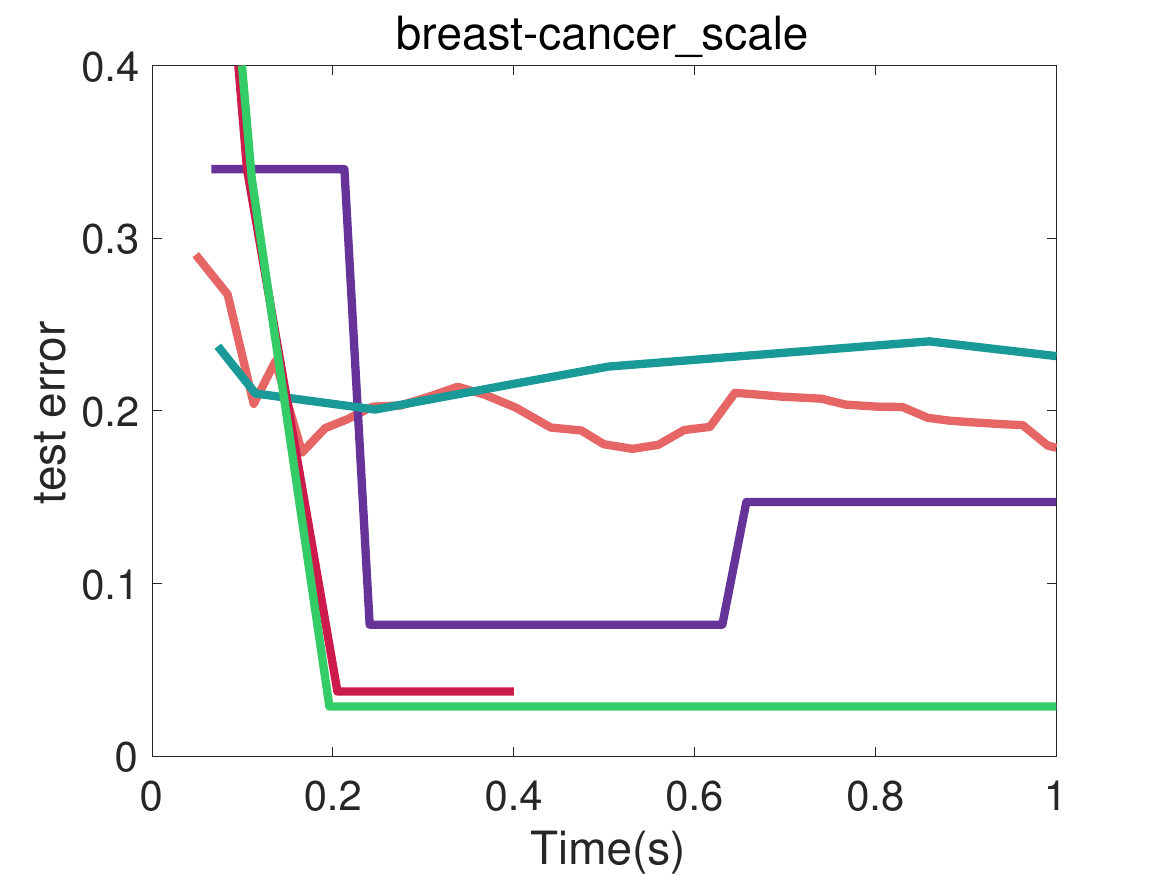}
 		\includegraphics[width=0.32\linewidth]{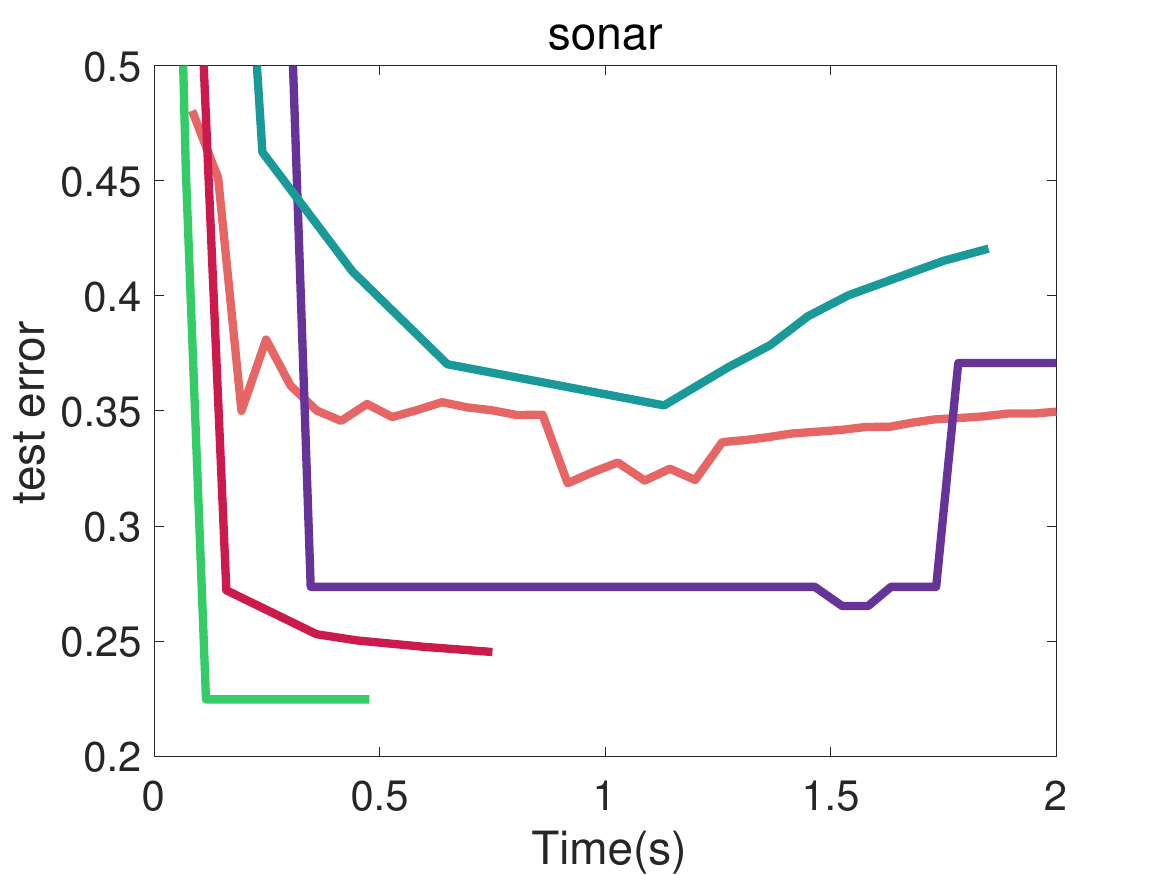}
 		\label{SVM4_test}
        \includegraphics[width=0.32\linewidth]{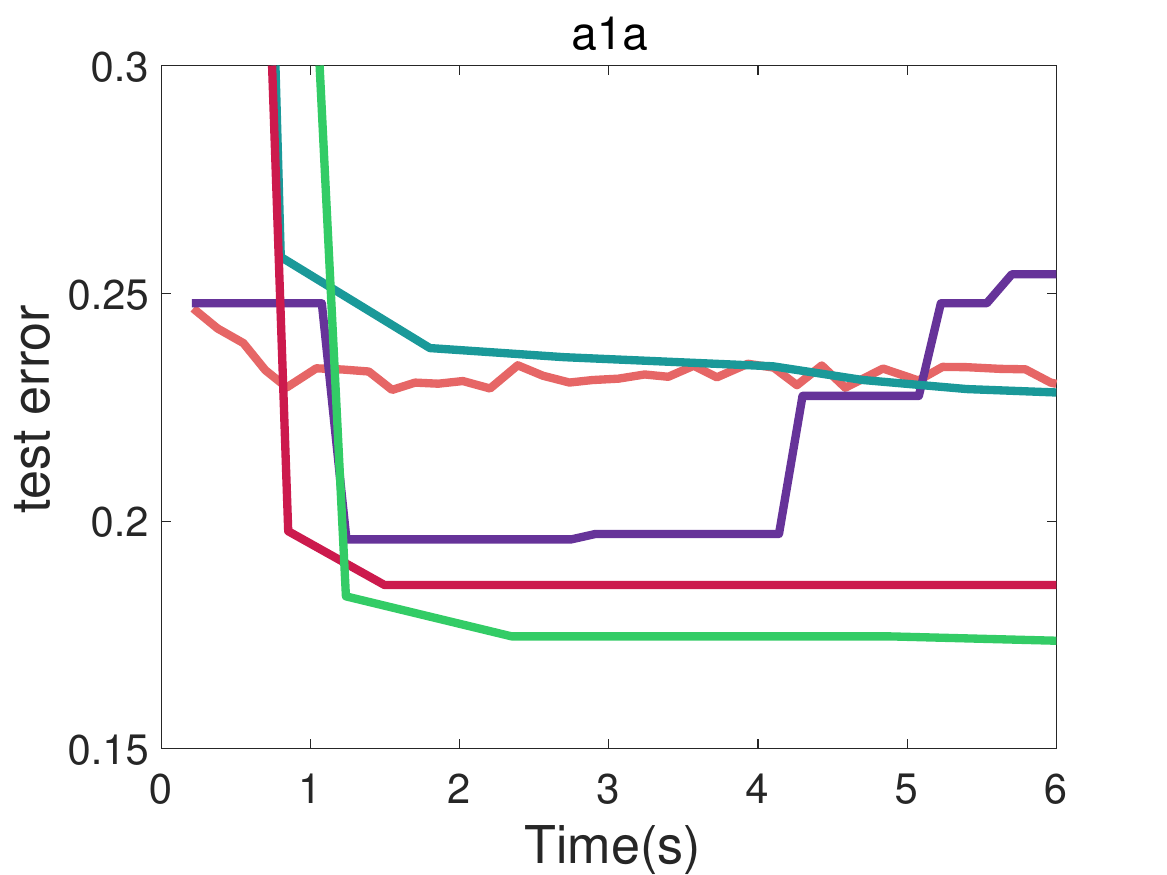}
        \label{SVM5_test}
 		\includegraphics[width=0.32\linewidth]{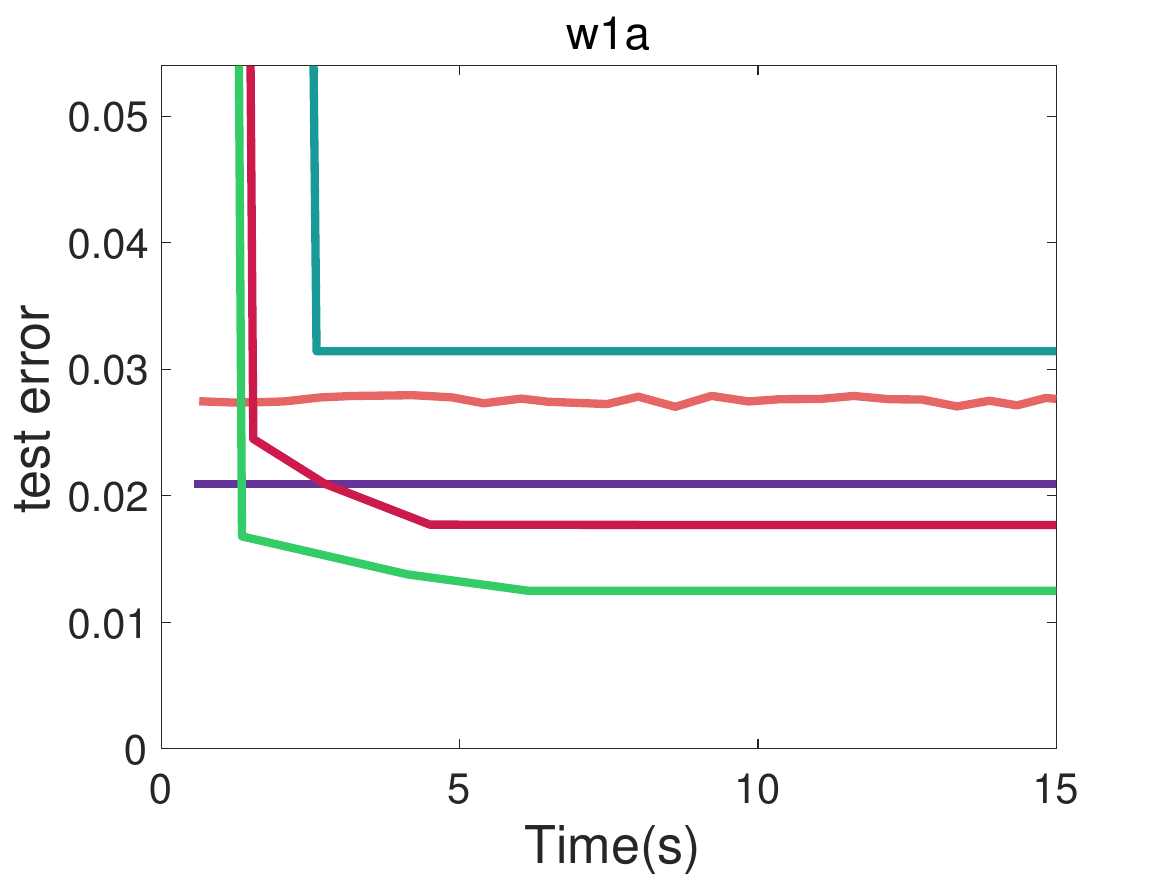}
 	\caption{Comparison for variation trend of validation error and test error with time in SVM experiments.}
 \label{fig:SVM}
 \vspace{-0.1in}
\end{figure*}
\begin{table*}
\caption{Elastic net problem on datasets gisette and sensit, where $|I_{tr}|$, $|I_{val}|$, $|I_{te}|$ and $p$ represent the number of training samples, validation samples, test samples and features, respectively.}
\centering
 \setlength\tabcolsep{2pt}
 \resizebox{\linewidth}{!}{
\scalebox{1}{
\begin{tabular}{clccc|cccc}
\toprule
\textbf{Settings} & \textbf{Methods} & \textbf{Time(s)} & \textbf{Val. Err.} & \textbf{Test Err.}&\textbf{Settings}  & \textbf{Time(s)} & \textbf{Val. Err.} & \textbf{Test Err.}\\
\midrule
 \multirow{6}{*}{\makecell[l]{gisette \\ $p=5000$ \\$|I_{tr}|=50$\\ $|I_{val}|=50$\\ $|I_{te}|=5900$}} & Grid & $36.14\pm4.52$ & $0.22\pm0.04$ & $0.23\pm0.01$ &\multirow{6}{*}{\makecell[l]{sensit \\ $p=78823$ \\ $|I_{tr}|=25$\\ $|I_{val}|=25$\\ $|I_{te}|=50$}}& $1.39\pm0.15$ & $1.40\pm0.79$ & $1.31\pm0.45$\\
 & Random & $55.44\pm9.21$ & $0.22\pm0.05$ & $0.23\pm0.03$ & &$1.28\pm0.10$ & $1.52\pm0.55$ & $1.45\pm0.40$ \\
 & TPE & $40.01\pm7.04$ & $0.22\pm0.05$ & $0.24\pm0.02$ && $1.78\pm0.09$ & $1.38\pm0.96$ & $1.39\pm0.55$\\
 & IGJO & $6.15\pm1.54$ & $0.24\pm0.05$ & $0.24\pm0.03$ && $0.49\pm0.74$ & $0.52\pm0.20$ & $0.61\pm0.11$\\
 & VF-iDCA & $5.38\pm1.65$ & $0.00\pm0.00$ & $0.19\pm0.01$ && $0.16\pm0.08$ & $0.23\pm0.11$ & $0.51\pm0.06$\\
 & LDMMA & $5.64\pm0.94$ & $0.00\pm0.00$ & $0.17\pm0.01$ \vspace{2pt}&&$0.16\pm0.07$ & $0.25\pm0.12$ & $0.40\pm0.09$\\
\bottomrule
\end{tabular}}}\label{table:high}
\end{table*}
\subsection{Experiments on Synthetic Data}
The synthetic data consists of observation matrix samples from specific distribution and response vectors with rational noise. Detailed descriptions of the synthetic data generation settings and parameter settings of each method are in \cref{detail-data}.
\subsubsection{Elastic Net}
The numerical results on elastic net are reported in Table \ref{table-ela}. We conduct 30 repeated experiments in each data size and take the average. Overall, LDMMA achieves the highest solution quality in the shortest running time on this problem model. Traditional gradient-free methods (grid search, random search, and TPE) still suffer from limitations in testing error and expensive time costs. Gradient-based methods IGJO perform slightly better on accuracy and efficiency, and VF-iDCA is the best among existing methods in the literature. Furthermore, LDMMA achieves more exquisite validation and test error beyond the reach of other methods along with greatly reduced time cost.
\subsubsection{Sparse Group Lasso}
We conduct experiments with different data scales and report numerical results averaged over 30 repetitions in Table \ref{table-sgl}. For each experiment, the generated datasets consist of $n$ training, $n/3$ validation, and 100 test samples. LDMMA still performs the best in the sense that it achieves the minimum time cost and test error, meanwhile a similar validation error with VF-iDCA. As the dimension of data increases, our methods appear to offer all-round fitness for the problem, which indicates the superiority of LDMMA in large-scale hyperparameter optimization. It is worth noting that our algorithm can obtain the optimal solution for both hyperparameters and upper-level variables by solving problem \eqref{eq:subeps}. This is a significant advantage of our algorithm.
\subsection{Experiments on Real Data}
\subsubsection{Support Vector Machine with Cross-validation}
We conduct experiments for support vector machine (SVM) model on real-world datasets. Real-world datasets tend to be lager in size than synthetic datasets and exhibit more complex and irregular sample distributions. Consequently, hyperparameter selection will be heavily influenced by the partition of the training, validation, and test sets. Different partition can lead to substantial variations in the predictive performance of the models. Therefore, we perform 3-fold and 6-fold cross-validation using six moderately sized real datasets: liver-disorders, diabetes, breast-cancer, sonar, a1a \citep{asuncion2007uci}, and w1a \citep{catanzaro2008fast}. These datasets are derived from medical statistics and offer rich features and samples for analysis.

The details of corresponding subproblem for cross-validation with dataset partition and experimental settings are presented in Appendix \ref{detail-data}. We report numerical results on three datasets in Table \ref{table-SVM-3fold-6fold} and \cref{fig:SVM}. As shown in Table \ref{table-SVM-3fold-6fold}, comparison results demonstrate that LDMMA consistently outperforms other optimization algorithms in terms of both the validation error and test error (except the case of diabetes-scal with 3-fold). Moreover, LDMMA achieves faster convergence than other methods.
\cref{fig:SVM} reports the variation trend of validation error and test error versus time from our experiments with 6-fold cross-validation.
We emphasize that LDMMA remarkably reduces the validation and test errors at a faster speed than other algorithms. These results verify the superiority and applicability of our algorithm for SVM on real-world datasets.

\subsubsection{Elastic Net with High Dimensional datasets}
Furthermore, to certify the robustness of our algorithm, it is necessary to conduct experiments with larger scale which may capture more practical settings. We consider elastic net problem on high dimendional datasets gisette\citep{guyon2004result} and sensit \citep{duarte2004vehicle}. Experimental results are reported in Table \ref{table:high}, demonstrating that even in relatively high dimensional problems, LDMMA still achieves competitive performance at a fast speed.

\section{Conclusion}
In this paper, we propose a novel single-level reformulation for a group of hyperparameter optimization problems, where the main steps are leveraging the structure of the lower-level problem and applying Fenchel's duality. Our reformulation does not involve complex implicit functions but conjugates of some atom functions. Based on the new reformulation, we then propose the LDMMA, which applies the majorization-minimization method to obtain a convex subproblem. One superiority of our method is that for many practical problems, our subproblem are conic programs so that the subproblem can be efficiently solved by the off-the-shelf solvers.  Theoretically,  we prove the sequence convergence of the LDMMA. Numerical experiments on both synthetic and real-world data demonstrate the outperformance of LDMMA over existing methods.

We remark that the methods for solving subproblem \eqref{eq:subeps} are not limited to the off-the-shelf solvers.  In future work, we will explore first-order methods that are suitable for high-dimension settings for solving subproblem \eqref{eq:subeps}; see, e.g., \cite{lan2011primal} and \cite{necoara2019complexity}.

\subsection*{Acknowledgements}
\label{Acknowledgements}
Rujun Jiang is partly supported by the National Key RD Program of China under grant 2023YFA1009300, National Natural Science Foundation of China under grants 12171100 and 72394364, and Natural Science Foundation of Shanghai 22ZR1405100. Anthony Man-Cho So is partly supported by the Hong Kong Research Grants Council (RGC) General Research Fund (GRF) project CUHK 14204823.

\bibliography{ref}

\section*{Checklist}
 \begin{enumerate}

 \item For all models and algorithms presented, check if you include:
 \begin{enumerate}
   \item A clear description of the mathematical setting, assumptions, algorithm, and/or model. [Yes]
   \item An analysis of the properties and complexity (time, space, sample size) of any algorithm. [Yes]
   \item (Optional) Anonymized source code, with specification of all dependencies, including external libraries. [Yes]
 \end{enumerate}

 \item For any theoretical claim, check if you include:
 \begin{enumerate}
   \item Statements of the full set of assumptions of all theoretical results. [Yes]
   \item Complete proofs of all theoretical results. [Yes]
   \item Clear explanations of any assumptions. [Yes]     
 \end{enumerate}

 \item For all figures and tables that present empirical results, check if you include:
 \begin{enumerate}
   \item The code, data, and instructions needed to reproduce the main experimental results (either in the supplemental material or as a URL). [Yes]
   \item All the training details (e.g., data splits, hyperparameters, how they were chosen). [Yes]
         \item A clear definition of the specific measure or statistics and error bars (e.g., with respect to the random seed after running experiments multiple times). [Yes]
         \item A description of the computing infrastructure used. (e.g., type of GPUs, internal cluster, or cloud provider). [Yes]
 \end{enumerate}

 \item If you are using existing assets (e.g., code, data, models) or curating/releasing new assets, check if you include:
 \begin{enumerate}
   \item Citations of the creator If your work uses existing assets. [Yes]
   \item The license information of the assets, if applicable. [Yes]
   \item New assets either in the supplemental material or as a URL, if applicable. [Yes]
   \item Information about consent from data providers/curators. [Yes]
   \item Discussion of sensible content if applicable, e.g., personally identifiable information or offensive content. [Not Applicable]
 \end{enumerate}

 \item If you used crowdsourcing or conducted research with human subjects, check if you include:
 \begin{enumerate}
   \item The full text of instructions given to participants and screenshots. [Not Applicable]
   \item Descriptions of potential participant risks, with links to Institutional Review Board (IRB) approvals if applicable. [Not Applicable]
   \item The estimated hourly wage paid to participants and the total amount spent on participant compensation. [Not Applicable]
 \end{enumerate}

 \end{enumerate}
\appendix
\onecolumn
The appendix is organized as follows.  In  Appendix \ref{appen:th},  we provide missing proofs of Sec. \ref{sec:th}. In Appendix. \ref{sec:a}, we use several widely used machine learning models to illustrate that the constraints in \eqref{eq:refsub} can often be represented by conic inequalities. In Appendix \ref{detail-data}, we provide details of our numerical experiments and give some additional experiments results. 


\section{Proofs in Section \ref{sec:th}}\label{appen:th}
\subsection{Proof for Lemma \ref{le:cq}}
\begin{proof}
   (i) The first item follows from \cite{ye2022difference}.

   (ii) The second item can be proved by contradiction. If NNAMCQ fails at some feasible point $\mathbf{z}:=(\mathbf{x},\boldsymbol{\lambda},\boldsymbol{r},\mathbf{\rho})$, i.e., there exists vector $(\mu_0,\mu_1,\ldots,\mu_{\tau})$ such that $0\in \mu_0\partial g(\mathbf{z})+\sum_{i=1}^{\tau}\mu_i\partial h_j(\mathbf{z})$, one can obtain a contradiction whether the constraint $g(\mathbf{z})=F(\mathbf{x},\boldsymbol{\lambda},\mathbf{\rho})+\sum_i{\lambda_ir_i}-\epsilon\leq0$ is active or not.
\begin{enumerate}
\item When $F(\mathbf{x},\boldsymbol{\lambda},\mathbf{\rho})+\sum_i{\lambda_ir_i}-\epsilon<0$, $\mu_0=0$ and we observe that $\partial h_i(\mathbf{z})=\partial\left(P_i(\mathbf{x})-r_i\right)$, $i\in[\tau]$ are linearly independent due to term $r_i$. Then $0\in \mu_0\partial g(\mathbf{z})+\sum_{i=1}^{\tau}\mu_i\partial h_j(\mathbf{z})$ implies $\mu_i=0$ for $i\in[\tau]$, which contradicts the definition of NNAMCQ that requires $\boldsymbol{\mu}\neq0$.
\item When $F(\mathbf{x},\boldsymbol{\lambda},\mathbf{\rho})+\sum_i{\lambda_ir_i}-\epsilon=0$ is active, the condition $0\in \mu_0\partial g(\mathbf{z})+\sum_{i=1}^{\tau}\mu_i\partial h_j(\mathbf{z})$ implies $\mathbf z={\rm argmin}_z F(\mathbf{x},\boldsymbol{\lambda},\mathbf{\rho})+\sum_i{\lambda_ir_i}$. This implies $F(\mathbf{x},\boldsymbol{\lambda},\mathbf{\rho})+\sum_i{\lambda_ir_i}=0$ because the left hand side of the constraint in \eqref{eq:fenchel} is always larger than or equal to 0 due to strong duality of the lower-level problem. This contradicts  $F(\mathbf{x},\boldsymbol{\lambda},\mathbf{\rho})+\sum_i{\lambda_ir_i}-\epsilon=0$.
\end{enumerate}
\end{proof}
\subsection{Proof for Lemma \ref{le:suff}}
\begin{proof}    We first prove that for $k\ge0$, $\z^{k}$ is a feasible point of \eqref{eq:subeps} for all $k\in\N$ by induction. Suppose this statement holds for some $k=d\ge0$, we prove it holds for $k=d+1$. Note that $(\x^{d+1},\blam^{d+1},\r^{d+1},\bro^{d+1})$ is optimal for problem \eqref{eq:subeps} with $\r^{d}$ and $\blam^{d}$.
     From the optimality of $\z^{d+1}$ to \eqref{eq:subeps}, we have
     \[h_j(\z^{d+1})\leq0,~\forall~j\in J~\text{ and }~\bar g^{d}(\z^{d+1})\leq0.\]
    Since $\bar g^{d}(\z^{d+1})\geq g (\z^{d+1})=\bar g^{d+1} (\z^{d+1})$ due to \cref{le:direct}, it follows that $\bar g^{d+1}(\z^{d+1})\leq0$ and thus $\z^{d+1}$ is feasible for problem \eqref{eq:subeps} with $k=d+1$. Finally, $\z^0$ is obviously feasible for \eqref{eq:subeps} with $k=0$ due to our special choice of $\blam^0$ and $\r^0$.

    Note that $(\x^{k+1},\blam^{k+1},\r^{k+1},\bro^{k+1})$ is optimal for problem \eqref{eq:subeps}. By the optimality of $\z^{k+1}$ and the feasibility of $\z^k$, we have
    \[L(\x^{k+1})+\frac{\beta}{2}\|\z^{k+1}-\z^k\|^2\leq L(\x^{k}),\]
    which proves (i). (Here $\x^0$ can be any feasible point.) Item (ii) directly follows from (i) and that $L$ is bounded below.
\end{proof}
\subsection{Proof for Theorem \ref{th:main}}
\begin{proof}
   (i) Assume subsequence $\{\z^{k_j}\}_{j\in\N}$ converges to $\z^*$.
     By Lemma \ref{le:cq}, the NNAMCQ holds at $\z^{k_j}$ for problem \eqref{eq:subeps} with 
     $k=k_j-1$. Moreover, the KKT optimality conditions implies that there exist acceptable Lagrange multipliers $\bmu^{k_j}=(\mu_0^{k_j},\mu_1^{k_j}\ldots,\mu_{\tau}^{k_j})$ such that \[\mu_0^{k_j}\bar g^{k_j}(\z^{k_j})=0, ~~\mu_i^{k_j}h_i(\z^{k_j})=0,~i=1,\ldots\tau\] and
     \begin{equation}\label{eq:partial_f}\small
         \bz\in \partial f^{k_j-1}(\z^{k_j})+\mu_0^{k_j}\partial \bar g^{k_j-1}(\z^{k_j})+\sum_{i=1}^{\tau}\mu_i^{k_j}\partial h_i(\z^{k_j})+N_{\cX}(\z^{k_j}).
     \end{equation}
     We then claim $\{\bmu^{l_j}\}_{j\in\N}$ is bounded.
      Since $\|\r^{l_j}-\r^{l_j-1}\|\rightarrow0$ and $\|\blam^{l_j}-\blam^{l_j-1}\|\rightarrow0$ by Lemma \ref{le:suff}, from $\z^{l_j}\rightarrow\z^*$ and \cref{def:maj} we have
\begin{equation}\label{eq:partial1}
    \partial \sum_{i=1}^\tau m(\lambda_i,r_i;\lambda_i^{l_j-1},r_i^{l_j-1}) \mid_{(\lambda_i^{l_j},r_i^{l_j})}\to\partial \sum_{i=1}^\tau\lambda_ir_i\mid_{(\lambda_i^{*},r_i^{*})}.
\end{equation}
According to Exercise 8.8(c) in \cite{rockafellar2009variational} and  the expressions of $\bar g^k $ and $g$, \eqref{eq:partial1} yields
     \[\limsup_{j\to\infty}\partial \bar g^{l_j-1}(\z^{l_j})\subseteq\partial g (\z^*).\]
     If $\{\bmu^{l_j}\}_{j\in N}$ is unbounded,
     by passing to a further subsequence if necessary, we assume$\bmu^{l_j}/\|\bmu^{l_j}\|\to\bar\bmu$. Note that $\|\bmu^{l_j}\|\to\infty$ and $\partial f(\z^{l_j})$ is bounded.
     Multiplying $\frac1{\|\bmu^{l_j}\|}$ on the right-hand side of \eqref{eq:partial_f} and taking limit to \eqref{eq:partial_f}, we have
     \[
     \begin{array}{l}
     \bz\in \bar{\mu}_0\partial g(\z^*) + \sum_{i=1}^m\bar{\mu}_i\partial h_i(\z^*)+N_{\cX}(\z^*)\text{ with }\|\bar{\bmu}\|=1,\\
 \bar\mu_0  g(\x^*)=0,~\bar \mu_i h_i(\x^*)=0,i\in[\tau],~\bar\mu_i\geq0,~i=0,\ldots,\tau.
         \end{array}
     \]
which contradicts the NNAMCQ at $\z^*$. Therefore, $\{\bmu^{l_j}\}_{j\in N}$ is bounded.

     Then we assume $\bmu^{k_j}\rightarrow\bmu^*$ by passing to a subsequence if necessary. Note that $\limsup_{j\to\infty}\partial f^{k_j}(\z^{k_j})\subseteq\partial f(\z^*)$ also holds because of the expression of $f^k$ and the outer semi-continuity of $\partial L(\x)$ (see, e.g., Definition 5.4 and Proposition 8.7 in \cite{rockafellar2009variational}).
  By $\z^{k_j}\rightarrow\z^*$, $\limsup_{j\to\infty}\partial \bar g^{k_j-1}(\z^{k_j})\subseteq\partial g (\z^*)$ and $\limsup_{j\to\infty}\partial f^{k_j}(\z^{k_j})\subseteq\partial f(\z^*)$, \eqref{eq:partial_f} yields
     \[\bz\in \partial f(\z^*)+\mu_0^*\partial g (\z^*)+\sum_{i=1}^{\tau}\mu_i^{*}\partial h_i(\z^{*})+N_{\cX}(\z^*),\]
     which says that $\z^*$ is a KKT point of \eqref{eq:blpeps}.

(ii) We construct a simple merit function for our model, which is defined by
$$
G(\z)=L(\mathbf{x})+\delta _{\mathcal{X}}(\mathbf{z})+\delta _{\{\mathbf{z}:g(\mathbf{z})\leq0\}}(\mathbf{z})+\sum _{j=1}^{\tau}\delta _{\{\mathbf{z}:h _j(\mathbf{z})\leq0\}}(\mathbf{z}).
$$
By the feasibility of $\z^{k+1}$ in \eqref{eq:subeps} and Definition \ref{def:maj}, $h_j(\z^{k+1})\leq0$, $j\in[\tau]$ and $g(\z^{k+1})\leq \bar{g}^k(\z^{k+1})\leq0$. Hence,
\[G\left(\z^{k+1}\right)=L(\x^{k+1})\text{ for all }k\in\N.\]
This, together with $L(\x^{k+1})\leq L(\x^k)-\frac{\beta}{2}\left\|\z^{k+1}-\z^k\right\|^2$ given by Lemma \ref{le:suff}, yields
\begin{equation}\label{eq:G_suff}
    G\left(\z^{k+1}\right)\leq G\left(\z^{k}\right)-\frac{\beta}{2}\left\|\z^{k+1}-\z^k\right\|^2.
\end{equation}
We then establish the relative error, which needs to estimate $\dist(\bz,\partial G(\z))$. To begin with, by the convexity of $L(\x)$ and  $P_i(\x)$, $L(\x)$ and $h_j(\z)$, $j\in[\tau]$ are regular. By the convexity of $F$ and weak convexity of $\sum_{i=1}^{\tau}\lambda_ir_i$, $g(\z)$ is regular due to Proposition 4.5 in \cite{vial1983strong}. Then, $\delta _{\{\mathbf{z}:g(\mathbf{z})\leq0\}}(\mathbf{z})$ and $\delta _{\{\mathbf{z}:h _j(\mathbf{z})\leq0\}}(\mathbf{z})$, $j\in[\tau]$ are regular by Exercise 8.14 in \cite{rockafellar2009variational}. It follows from Corollary 10.9 in \cite{rockafellar2009variational} that 
\[\partial G(\z)=\partial L(\x)+\partial\delta_{\cX}(\z)+\partial\delta_{\{\mathbf{z}:g(\mathbf{z})\leq0\}}(\mathbf{z}) +\sum_{j=1}^{\tau}\partial\delta _{\{\mathbf{z}:h _j(\mathbf{z})\leq0\}}(\mathbf{z}).\]
Since Exercise 8.14 in \cite{rockafellar2009variational} ensures that $\partial\delta_{\cX}(\z)=N_{\cX}(\z)$, $\partial\delta_{\{\mathbf{z}:g(\mathbf{z})\leq0\}}(\mathbf{z})=N_{\{\mathbf{z}:g(\mathbf{z})\leq0\}}(\z)$, and $\partial\delta _{\{\mathbf{z}:h _j(\mathbf{z})\leq0\}}(\mathbf{z})=N_{\{\mathbf{z}:h _j(\mathbf{z})\leq0\}}(\mathbf{z})$, $j\in[\tau]$, we have
\[ \partial G(\z)= \partial L(\x)+N_{\cX}(\z)+N_{\{\mathbf{z}:g(\mathbf{z})\leq0\}}(\mathbf{z}) +\sum_{j=1}^{\tau}N_{\{\mathbf{z}:h _j(\mathbf{z})\leq0\}}(\mathbf{z}).\]
By Corollary 10.50 in \cite{rockafellar2009variational}, it follows that $\partial G(\z)=\nabla L(\x)+N_{\cX}(\z)+\mu_0\partial g(\mathbf{z})(\mathbf{z}) +\sum_{j=1}^{\tau}\mu_i\partial h _j(\mathbf{z})$, where $\mu_0,\mu_j\geq0$, $\mu_0g(\z)=0$, and $\mu_jh_j(\z)=0$ for $j\in[\tau]$. Thus,
\begin{equation}\label{eq:subdiff_G}
\begin{array}{cc}
     \partial G(\z^{k})=&\{\partial L(\x^k)+N_{\cX}(\z^k)+\mu^k_0\partial g(\mathbf{z}^k) +\sum_{j=1}^{\tau}\mu^k_i\partial h _j(\mathbf{z}^k):\\
    &\quad\mu^k_0,\mu^k_j\geq0,\mu^k_0g(\z^k)=0,\mu^k_jh_j(\z^k)=0 \text{ for }j\in[\tau]\}.
\end{array}
\end{equation}
Note that $\partial f^{k-1}
(\z^k)=\partial L(\x^k)+\beta(\z^k-\z^{k-1})$ and that Definition \ref{def:maj} implies \[\partial  \bar{g}^{k-1}(\z^k)=\partial g(\z^k)+\sum_{i=1}^{\tau}\left(\nabla m(\lambda_i^k,r_i^k;\lambda_i^{k-1},r_i^{k-1})-\nabla m(\lambda_i^k,r_i^k;\lambda_i^{k},r_i^{k})\right).\]
Comparing \eqref{eq:subdiff_G} and the optimality conditions \eqref{eq:partial_f}, we have \[\bz\in \partial G(\z^k)+\beta(\z^k-\z^{k-1})+\mu_0^k\sum_{i=1}^{\tau} (\nabla m(\lambda_i^k,r_i^k;\lambda_i^{k-1},r_i^{k-1})-\nabla m(\lambda_i^k,r_i^k;\lambda_i^{k},r_i^{k})),\]
where $\mu_0^k$ is an acceptable Lagrange multiplier.
This is equivalent to
\[\beta(\z^{k-1}-\z^k)+ \sum_{i=1}^{\tau}\left(\nabla m(\lambda_i^k,r_i^k;\lambda_i^{k},r_i^{k})-\nabla m(\lambda_i^k,r_i^k;\lambda_i^{k-1},r_i^{k-1})\right)\in \partial G(\z^k).\]
By the triangle inequality, it follows that
\begin{equation}\label{eq:distGk}
\begin{aligned}
     &\dist\left(\bz,\partial G(\z^k)\right)\\
     \leq& \beta\left\|\z^{k-1}-\z^k\right\|+\mu_0^k\sum_{i=1}^{\tau}\left\|\nabla m(\lambda_i^k,r_i^k;\lambda_i^{k},r_i^{k})-\nabla m(\lambda_i^k,r_i^k;\lambda_i^{k-1},r_i^{k-1}) \right\|.
\end{aligned}
\end{equation}
Recall that we have proved that sub-sequence $\{\mu^{k_j}\}_{j\in\N}$ is bounded in (i) by the NNAMFCQ if $\{\z^{k_j}\}_{j\in\N}$ converges. This, together with the boundedness of $\{\z^k\}_{k\in\N}$ ensures that $\{\mu^k\}_{k\in\N}$ is bounded. Thus, 
\begin{equation}\label{eq:mu0kM}
\mu_0^k\leq M_1\text{ for some }    M_1>0.
\end{equation}
Furthermore, by the local Lipschitz continuity of $\nabla m(\xi,\zeta;\bar\xi,\bar\zeta)$ w.r.t. $(\bar\xi,\bar\zeta)$ given by Definition \ref{def:maj}, for $\{\z^k\}_{k\in\N}$, there exists $M_2>0$ such that 
\begin{equation}\label{eq:mrk}
\begin{aligned}
      \|\nabla m(\lambda_i^k,r_i^k;\lambda_i^{k},r_i^{k})-\nabla m(\lambda_i^k,r_i^k;\lambda_i^{k-1},r_i^{k-1}) \|&\leq M_2\|(\lambda_i^k-\lambda_i^{k-1},r_i^k-r_i^{k-1})\|\\
      &\leq M_2\left\|\z^{k-1}-\z^k \right\|,
\end{aligned}
\end{equation}
where the second inequality directly follows from $\z=(\x,\blam,\r,\bro)$. Combining \eqref{eq:distGk}, \eqref{eq:mu0kM}, and \eqref{eq:mrk}, we have
\begin{equation}\label{eq:relative}
    \dist\left(\bz,\partial G(\z^k)\right)
     \leq (\beta+\tau M_1M_2)\left\|\z^{k-1}-\z^k\right\|.
\end{equation}
Next, we verify that $G$ satisfies the Kurdyka-Lojasiewicz (KL) property. It suffices to show that $G$ is a semi-algebraic function by \cite{bolte2007lojasiewicz}. Note that it has been shown by \cite{attouch2010,attouch2013convergence,facchinei2003finite,bolte2014proximal} that the semi-algebraic property is preserved under many operations such as finite sum, product, and partial maximization operations. Moreover, the epi-graphs of semi-algebraic functions are semi-algebraic sets and indicator functions of semi-algebraic sets are semi-algebraic functions. As a direct consequence, $L(\x),g(\z)$, and $h_j(\z),j\in[\tau]$ are all semi-algebraic functions. $\cX,\{\z:g(\z)\leq0\}$, and $\{\z:h_j(\z)\leq0\}$, $j\in[\tau]$ are semi-algebraic sets and $\delta_{\{\z:g(\z)\leq0\}},\delta_{\{\z:h_j(\z)\leq0\}},j\in[\tau]$ are semi-algebraic functions. Hence, $G$ is semi-algebraic and satisfies the KL property.

Finally, combining \eqref{eq:G_suff}, \eqref{eq:relative}, and that $G$ satisfies the KL property, Theorem 2.9 in \cite{attouch2013convergence} implies the convergence of $\{\z^k\}_{k\in\N}$ and that the limiting point $\z^*$ is a stationary of $G$, i.e., $\bz\in\partial G(\z^*)$. By \eqref{eq:subdiff_G}, this is equivalent to that $\z^*$ is a KKT point of \eqref{eq:blpeps}. The proof is complete.

\end{proof}
\section{Reformulations and Subproblems of LDMMA for Different Models}\label{sec:a}
In this section, we consider several widely used machine learning models to illustrate that the constraints in \eqref{eq:refsub} can often be represented by conic inequalities. We choose $m(\xi,\zeta;\bar\xi,\bar\zeta) = \frac{1}{2}\left(\frac{\bar\xi}{\bar\zeta}\zeta^2+\frac{\bar\zeta}{\bar\xi}\xi^2\right)$ to demonstrate the tractability of our subproblems when $\r^k,\blam^k>0$\footnote{We note that for the case that $r^k_i=0$ or $\lbd^k_i=0$, we can choose $m$ by \eqref{eq:dca_m} and conclude similar results.}.
\subsection{Elastic Net and Sparse Group Lasso}
We recall the elastic net problem
\begin{equation*}\label{eq:elastic_appen}
    \begin{array}{ll}
         \min\limits_{\x} & L(\x)=\frac12\|A_{val}\x-\b_{val}\|^2\\
         {\rm s.t.} & \x \in{\argmin}\frac12\|A_{tr}\x-\b_{tr}\|^2+\lambda_1\|\x\|_1+\frac{\lambda_2}{2}\|\x\|_2^2
    \end{array}
\end{equation*}
and the sparse group lasso problem
\begin{equation*}\label{eq:sparse}
    \begin{array}{ll}
         \min\limits_{\x} & L(\x)=\frac12\|A_{val}\x-\b_{val}\|^2\\
         {\rm s.t.} & \x \in{\argmin}\frac12\|A_{tr}\x-\b_{tr}\|^2+\lambda_{M+1}\|\x\|_1+\sum\limits_{i=1}^M\lambda_i\|\x^{(i)}\|_2,
    \end{array}
\end{equation*}
where $\x^{(i)}$ is a sub-vector of $\x$.

The following two propositions show the explicit formulation of \eqref{eq:refsub}, which are closely related to \eqref{eq:subeps}, subproblems of the proposed LDMMA.
\begin{proposition}
    For the elastic net problem with linearly independent $\{\mathbf{a}_i:i\in I_{tr}\}$, let $A_{tr}=(\mathbf{a}_i)_{i\in I_{tr}},\b_{tr}=(b_i)_{i\in I_{tr}}$ denote the training data and $A_{val}=(\mathbf{a}_i)_{i\in I_{val}},\b_{val}=(b_i)_{i\in I_{val}}$ denote the validation data, then \eqref{eq:refsub}  can be reformulated into the following conic program:
     \begin{equation*}
    \begin{array}{lcl}
     &\min & t\\
    &{\rm s.t.} & \left\|\begin{pmatrix}
             A_{tr}\x-\b_{tr},
         \sqrt{\frac{{\lbd}^0_1}{r^0_1}}r_1,
        \sqrt{\frac{ {r}^0_1}{ {\lbd}^0_1}}\lbd_1,
        \sqrt{\frac{{\lbd}^0_2}{r^0_2}}r_2,
        \sqrt{\frac{ {r}^0_2}{ {\lbd}^0_2}}\lbd_2,
        \w+\b_{tr},
      \frac{s}{\|\b_{tr}\|_2}
         \end{pmatrix}\right\|_2
         \leq
         \|\b_{tr}\|_2-\frac{1}{\|\b\|_2}s, \\
            & &  A_{tr}^T\w+ {\bro}_1+\bro_2=\bz,\\
  & & \|\x\|_1\leq r_1,\ \|\bro_1\|_{\infty}\leq\lambda_1,\\
   & &  \left\|\begin{pmatrix}
         \x\\
         r_2-\frac12
     \end{pmatrix}\right\|_2\leq r_2+\frac12,\
     \left\|\begin{pmatrix}
        \sqrt{2}\bro_2\\
        s-\lbd_2
    \end{pmatrix}\right\|_2\leq s+\lbd_2, \\
    & & \left\|\begin{pmatrix}
        A_{val}\x-\b_{val}\\
        t-\frac12
    \end{pmatrix} \right\|_2\leq t+\frac12.\\
    \end{array}
\end{equation*}
\end{proposition}
\begin{proposition}\label{pro:sgl}
    For the sparse group lasso problem with linearly independent $\{\mathbf{a}_i:i\in I_{tr}\}$, let $A_{tr}=(\mathbf{a}_i)_{i\in I_{tr}},\b_{tr}=(b_i)_{i\in I_{tr}}$ denote the training data and $A_{val}=(\mathbf{a}_i)_{i\in I_{val}},\b_{val}=(b_i)_{i\in I_{val}}$ denote the validation data, then \eqref{eq:refsub}  can be reformulated into the following conic program:
     \begin{equation*}
    \begin{array}{lcl}
     &\min & t\\
    &{\rm s.t.} &\left\|\begin{pmatrix}
             A_{tr}\x-\b_{tr},
         \sqrt{\frac{{\lbd}^0_1}{r^0_1}}r_1,
        \sqrt{\frac{ {r}^0_1}{ {\lbd}^0_1}}\lbd_1,\ldots
        \sqrt{\frac{{\lbd}^0_{M+1}}{r^0_{M+1}}}r_{M+1},
        \sqrt{\frac{ {r}^0_{M+1}}{ {\lbd}^0_{M+1}}}\lbd_{M+1},
        \w+\b_{tr}
         \end{pmatrix}\right\|_2
         \leq
         \|\b_{tr}\|_2,\\
            & & A_{tr}^T\w+\sum\limits_{i=1}^{M+1}\boldsymbol{\rho}_i=\bz,\\
 & & \|\x\|_1\leq r_{M+1},\ \|\bro_1\|_{\infty}\leq\lambda_{M+1},\\
    & & \|\x^{(i)}\|_2\leq r_{i},\ \|\bro_i\|_2\leq\lbd_i\ \ \text{for }i=1,2,\ldots,M,\\
    & & \left\|\begin{pmatrix}
        A_{val}\x-\b_{val}\\
        t-\frac12
    \end{pmatrix} \right\|_2\leq t+\frac12.\\
    \end{array}
\end{equation*}
\end{proposition}
As a comparison, the reformulation \eqref{eq:ref} for elastic net and sparse group lasso are respectively
\begin{equation}
    \begin{array}{lcl}
       &\min & \frac12\|A_{val}\x-\b_{val}\|_2^2 \\
    &{\rm s.t.} &
   \frac12\|A_{tr}\x-\b_{tr}\|_2^2 +\frac12\|\w+\b\|_2^2-\frac12\|\b\|_2^2+ \lambda_1r_1+\lambda_2r_2+\frac1{2\lambda_2}\|\bro_2\|_2^2 \leq 0,\\
  && A^T\w+\bro_1+\bro_2=\bz,\\
  &&  \|\x\|_1\leq r_1,\ \frac12\|\x\|_2^2\leq r_2,\|\bro\|_{\infty}\leq\lambda_1 ,  \quad \blam\ge0\\
    \end{array}
\end{equation}
and
\begin{equation}
    \begin{array}{lcl}
       &\min & \frac12\|A_{val}\x-\b_{val}\|_2^2 \\
    &{\rm s.t.} &
   \frac12\|A_{tr}\x-\b_{tr}\|_2^2 +\frac12\|\w+\b\|_2^2-\frac12\|\b\|_2^2+ \sum_{i=1}^{M+1}\lambda_ir_i\leq 0,\\
  && A^T\w+\sum_{i=1}^{M+1}\bro_i=\bz,\\
  &&  \|\x\|_1\leq r_{M+1},\ \|\bro_{M+1}\|_{\infty}\leq\lambda_{M+1}\ \|\x^{(i)}\|_2\leq r_{i},\ \|\bro_{i}\|_2\leq \lambda_i,\ i=1,\ldots M,  \quad \blam\ge0.\\
    \end{array}
\end{equation}
We use a unified proof for the two problems. To this end, let $P_i$ denote $\|\cdot\|_1$, $\|\cdot\|_2$ or $\frac12\|\cdot\|^2_2$ of $\x^{(i)}$, where $\x^{(i)}$ is a sub-vector of $\x$. Further, we let
\[\begin{array}{l}
     [\tau]=J_1\cup J_2\cup J_3, \\
     P_{i_1}=\|\cdot\|_1,\ i_1\in J_1,\\
     P_{i_2}=\|\cdot\|_2,\ i_2\in J_2,\\
     P_{i_3}=\frac12\|\cdot\|^2_2,\ i_3\in J_3.
\end{array}
\]

\begin{lemma}\label{pro:nonconvex}
   \eqref{eq:refsub} is equivalent to
   {
    \begin{equation}\label{eq:nonconvex}
    \begin{array}{lcl}
     &\min & L(\x) \\
    &{\rm s.t.} &      l(\x)+\sum_{i=1}^{\tau}\frac{\frac{{\lbd}^k_i}{ {r}^k_i}r_i^2+\frac{{r}^k_i}{{\lbd}^k_i}\lbd_i^2}{2}+l^*(-\sum\limits_{i=1}^{\tau}{\bro}_i)+\sum\limits_{i\in J_3}s_i \leq 0 \\
    & & \|\x \|_1\leq r_i,\ i\in J_1,\ \|\x \|_2\leq r_i,\ i\in J_2 \\
    & & \left\|\begin{pmatrix}
         \x \\
    r_i-\frac12
     \end{pmatrix}\right\|_2\leq r_i+\frac12,\ i\in J_3\\
 &  & \|\bro_i\|_{\infty}\leq\lambda_i\text{ for }i\in J_1,\ \|\bro_i\|_{2}\leq\lambda_i\text{ for }i\in J_2\\
   & & \left\|\begin{pmatrix}
        \sqrt{2}\bro_i\\
        s_i-\lbd_i
    \end{pmatrix}\right\|_2\leq s_i+\lbd_i\ \text{ for }i\in J_3. \\

    \end{array}
\end{equation}
}
\end{lemma}
\begin{proof}
We first simplify $F$ defined by \eqref{eq:f}.
    Note that \text{for }$i\in J_1$ \[\ P_i^*(\y)=\begin{cases}
        0,\text{ if } \|\y\|_{\infty}\leq1\\
        \infty\text{ otherwise},
    \end{cases}\] \text{for }$i\in J_2$\[ P_i^*(\y)=\begin{cases}
        0,\text{ if } \|\y\|_{2}\leq1\\
        \infty\text{ otherwise},
    \end{cases}\]
    and for $i\in J_3$ $P_i(\y)=\frac12\|\y\|_2^2$. For $i\in J_3$, we introduce $\lbd_iP_i^*(\frac{\rho_i}{\lbd_i})\leq s_i$. Then the constraints of \eqref{eq:refsub} amount to
        \[ \begin{cases}
          l(\x)+\sum_{i=1}^{\tau} \frac{\frac{{\lbd}^k_i}{ {r}^k_i}r_i^2+\frac{ {r}^k_i}{ {\lbd}^k_i}\lbd_i^2}{2}+l^*(-\sum\limits_{i=1}^{\tau}{\bro}_i)\leq 0\\
           \|\x \|_1\leq r_i,\ i\in J_1,\ \|\x \|_2\leq r_i,\ i\in J_2 \\
    \frac12\|\x \|_2^2\leq r_i,\ i\in J_3\\
          \|\bro_i\|_{\infty}\leq\lambda_i,\ i\in J_1,\ \|\bro_i\|_{2}\leq\lambda_i,\ i\in J_2,\\
           \frac{\|\bro_i\|_2^2}{\lambda_i}\leq2 s_i\,\ i\in J_3.
    \end{cases} \]
    By $s_i\lbd_i=\frac{(s_i+\lbd_i)^2-(s_i-\lbd_i)^2}{4}$, $2r_i=(r_i+\frac12)^2-(r_i-\frac12)^2$ and taking square root, $\frac{\|\bro_i\|_2^2}{\lambda_i}\leq2 s_i$ and $\frac12\|\x \|_2^2\leq r_i $ are further equivalent to second-order cone constraints
    \[\left\|\begin{pmatrix}
        \sqrt{2}\bro_i\\
        s_i-\lbd_i
    \end{pmatrix}\right\|_2\leq s_i+\lbd_i,\    \left\|\begin{pmatrix}
         \x \\
         r_i-\frac12
     \end{pmatrix}\right\|_2\leq r_i+\frac12\]
    which concludes the result.
\end{proof}
Many practical problems of interest choose the least square error as the loss function, i.e., $l(\x)=\frac12\|A\x-\b\|_2^2$. Particularly, we consider $A^T$ to be of full row rank, i.e., the feature vectors of the data are linearly independent. In this case, \eqref{eq:refsub} is further equivalent to a conically constrained convex problem.
\begin{lemma}
    \label{co:ls}
If $l(\x)=\frac12\|A\x-\b\|_2^2$ with $A\in\R^{m\times n}$ being of full row rank, \eqref{eq:refsub} can be further written as
  {
    \begin{equation}\label{eq:conic_ls}
    \begin{array}{lcl}
     &\min & L(\x) \\
    &{\rm s.t.} &
        { \left\|\begin{pmatrix}
             A\x-\b,
             \ldots,
         \sqrt{\frac{{\lbd}^k_i}{r^k_i}}r_i,
        \sqrt{\frac{ {r}^k_i}{ {\lbd}^k_i}}\lbd_i,
        \ldots,
        \w+\b,
      \frac{\sum\limits_{i\in J_3} s_i}{\|\b\|_2}
         \end{pmatrix}\right\|_2}\leq
         \|\b\|_2-\frac{1}{\|\b\|_2}\sum\limits_{i\in J_3} s_i \\
   &     &    A^T\w+ \sum\limits_{i=1}^{\tau}{\bro}_i=\bz\\
& & \|\x \|_1\leq r_i,\ i\in J_1,\ \|\x \|_2\leq r_i,\ i\in J_2 \\
  & &   \left\|\begin{pmatrix}
         \x \\
         r_i-\frac12
     \end{pmatrix}\right\|_2\leq r_i+\frac12,\ i\in J_3\\
  & & \|\bro_i\|_{\infty}\leq\lambda_i,\ i\in J_1,\ \|\bro_i\|_2\leq\lambda_i,\ i\in J_2\\
 &   & \left\|\begin{pmatrix}
        \sqrt{2}\bro_i\\
        s_i-\lbd_i
    \end{pmatrix}\right\|_2\leq s_i+\lbd_i,\ i\in J_3. \\
    \end{array}
\end{equation}
}
\end{lemma}
\begin{proof}
    We first compute $l^*(\y):=\max\limits_{\x}\ \y^T\x-\frac12\|A\x-\b\|_2^2$, which is
    \[l^*(\y)=\max\limits_{\x}\ -\frac12\x^TA^TA\x+\x^T\left(A^T\b+\y\right)-\frac12\|\b\|^2.\]
    By Example 9.1.1 in \cite{boyd2004convex}, if the above problem is solvable, $A^TA\x^*=A^T\b+\y$ is solvable with $\x^*$ being the optimal solution. This is equivalent to $\y=A^T\w$ for some $\w\in\R^m$ and then the full column rank of $A$ yields $A\x^*=\b+\w$. Substituting it into $l^*(\y)$, we have $l^*(\y)=\frac12\|\w+\b\|_2^2-\frac12\|\b\|_2^2$, hence
    \[l^*(\y)=\begin{cases}
       \frac12\|\w+\b\|_2^2-\frac12\|\b\|_2^2 \text{ if } \y=A^T\w\\
        \infty\text{ otherwise}.
    \end{cases}\]
    Then the first constraint of \eqref{eq:nonconvex} can be replaced with
    {
    \[\begin{cases}
        \|A\x-\b\|_2^2+\sum\limits_{i=1}^{\tau}\left\|\begin{pmatrix}
            \sqrt{\frac{{\lbd}^0_i}{r^0_i}}r_i\\
             \sqrt{\frac{ {r}^0_i}{ {\lbd}^0_i}}\lbd_i
        \end{pmatrix}\right\|^2_2+ \|\w+\b\|_2^2+2\sum\limits_{i\in J_3}s_i\leq\|\b\|_2^2\\
        A^T\w+ \sum\limits_{i=1}^{\tau}\bro_i=\bz.
    \end{cases}\]
    }
By using \[\|\b\|^2_2-2\sum\limits_{i\in J_3} s_i=\left(\|\b\|_2-\frac{\sum\limits_{i\in J_3} s_i}{\|\b\|_2}\right)^2- \left(\frac{\sum\limits_{i\in J_3} s_i}{\|\b\|_2}\right)^2\]
and taking the square root, we conclude the result.
\end{proof}

 Now we are ready to give proofs for Propositions \ref{pro:elastic_net} and \ref{pro:sgl}.
\begin{proof}
Based on the above propositions, we give the expressions of the subproblems of the elastic net and sparse group lasso. Note that the expression of \eqref{eq:ref} is very similar to that of \eqref{eq:refsub}, we omit the augments for simplicity.

For the elastic net problem, $J_1=\{1\}$ and $J_2=\emptyset$ and $J_3=\{2\}$. Then the conclusion follows from \cref{co:ls} and introducing the variable $t$ such that $t\geq L(\x)=\frac12\|A_{val}\x-\b_{val}\|_2^2$, which is equivalent to $\left\|\begin{pmatrix}
        A_{val}\x-\b_{val}\\
        t-\frac12
    \end{pmatrix} \right\|_2\leq t+\frac12$.

    For the sparse group lasso problem, we let $J_2=\{1,2,\ldots,M\}$ and $J_1=\{M+1\}$ and the augments are the same as that of \cref{pro:elastic_net}.
\end{proof}
\begin{remark}
   {  Without the linear independence of the data, we can still obtain a conic program for the subproblem but with one extra linear constraint.} In fact, the linear independence, i.e., the full column rank of $A^T$, is only used in the proof of Lemma \ref{co:ls}  to yield $A\mathbf{x}^*=\mathbf{b}+\mathbf{w}$ from $A^TA\mathbf{x}^*=A^T\mathbf{b}+\mathbf{y}$ and $\mathbf{y}=A^T\mathbf{w}$. Without this condition, we could have
$$
A\mathbf{x}^*=\mathbf{b}+\mathbf{w}+\mathbf{v}\text{ with } A^T\mathbf{v}=\mathbf{0},
$$
and the arguments of conic program reformulation still go through. 
\end{remark}
\subsection{Support Vector Machine}
We now consider the support vector machine (SVM) problem
\begin{equation}\label{eq:SVMprimal}
    \begin{array}{ll}
         \min\limits_{\w,c} & L(\w,c)=\sum\limits_{j\in I_{val}}\max(1-b_j(\w^\top \a_j-c),0) \\
         {\rm s.t.} & \w\in\underset{-\overline{\w}\leq\w\leq\overline{\w}}{\argmin}\sum\limits_{j\in I_{tr}}\max(1-b_j(\w^\top \a_j-c),0)+\frac{\lambda}{2}\|\w\|_2^2.
    \end{array}
\end{equation}
\begin{proposition}
For the SVM problem, let $B_{tr}=\operatorname{diag}\{b_j,j\in I_{tr}\}$ denote the diagonal matrix whose diagnal elements consist of $\{b_i:i\in I_{tr}\}$ in sequence. Let $A_{tr}=(\a_i)_{i\in I_{tr}}$ denote the training data (without labels), $\w,\r_2\in\R^p$ and $[r_2]_j$ denotes the $j$-th element of $\r_2$. Given $\overline{\w}^0,\r_2^0>0$, \eqref{eq:refsub} for problem \eqref{eq:SVMprimal} can be reformulated as
\begin{equation}\label{eq:svm1}
    \begin{array}{lcl}
        & \min & L(\w,c)=\sum\limits_{j\in I_{val}}\max(1-b_j(\w^\top \a_j-c),0) \\
        & {\rm s.t.} & \sum\limits_{j\in I_{tr}}\max(1-b_j(\w^\top \a_j-c),0)+\frac12(\frac{\lambda^0}{r_1^0}r_1^2+\frac{r_1^0}{\lambda^0}\lambda^2)+\frac12\sum\limits_{j=1}^p\left(\frac{\overline{w}_j^0}{[r_2^0]_j}{[r_2]_j}^2+\frac{[r_2^0]_j}{\overline{w}_j^0}\overline{w}_j^2\right)+s-\mathbf{1}^T \v\leq0,\\
       &&   \frac{1}{2}\|\w\|^2\leq r_1,\\
       &&   \left\|\begin{matrix}
             \sqrt{2}\rho\\ \lambda-s
         \end{matrix}\right\|_2\leq \lambda+s,\\
     &&     \left(\begin{matrix}
        A_{tr}^T\\ \mathbf{1}^T
        \end{matrix}\right)B_{tr}\mathbf{v}+\left(\begin{matrix}
        -\bro\\0
        \end{matrix}\right)+\left(\begin{matrix}
         {\boldsymbol{\alpha}}_2- {\boldsymbol{\alpha}}_1\\0
    \end{matrix}\right)=0,\\
        && \boldsymbol{\alpha}_1+\boldsymbol{\alpha}_2=\r_2,\\
     &&    0\leq \mathbf{v}\leq1,\boldsymbol{\alpha}_1,\boldsymbol{\alpha}_2\geq0,\\
         &&-\overline{\w}\leq\w\leq\overline{\w}.
    \end{array}
\end{equation}
Moreover, \eqref{eq:ref} for SVM is
\begin{equation}\label{eq:svm2}
    \begin{array}{lcl}
        & \min & L(\w,c)=\sum\limits_{j\in I_{val}}\max(1-b_j(\w^\top \a_j-c),0) ,\\
        & {\rm s.t.} & \sum\limits_{j\in I_{tr}}\max(1-b_j(\w^\top \a_j-c),0)+r_1\lambda+\r_2^\top\overline{\w}+s-\mathbf{1}^T \v\leq0,\\
      &&    \frac{1}{2}\|\w\|^2\leq r_1,\\
        &&  \left\|\begin{matrix}
             \sqrt{2}\bro\\ \lambda-s
         \end{matrix}\right\|_2\leq \lambda+s,\\
    &&      \left(\begin{matrix}
        A_{tr}^T\\ \mathbf{1}^T
        \end{matrix}\right)B_{tr}\v+\left(\begin{matrix}
        -\bro\\
        0
        \end{matrix}\right)+\left(\begin{matrix}
         {\boldsymbol{\alpha}}_2- {\boldsymbol{\alpha}}_1\\0
    \end{matrix}\right)=0,\\
    && \boldsymbol{\alpha}_1+\boldsymbol{\alpha}_2=\r_2,\\
   &&      0\leq \v\leq1,\boldsymbol{\alpha}_1,\boldsymbol{\alpha}_2\geq0,\\
         &&-\overline{\w}\leq\w\leq\overline{\w}.
    \end{array}
\end{equation}
\end{proposition}
We remark that \eqref{eq:svm1} is also equivalent to a conic program by the same augments of Lemma \ref{co:ls} and the proofs are omitted for the sake of brevity.
\begin{proof}
The lower-level problem can be written as
\begin{equation*}
    \begin{array}{ll}
        \min\limits_{\w,c,\z} & l(\w,c)+\lambda P(\z)  \\
        {\rm s.t.} & \w=\z,\quad g_1(\w,c)\leq0,\ g_2(\w,c)\leq0,
    \end{array}
\end{equation*}
where $l(\w,c)=\sum\limits_{j\in I_{tr}}\max(1-b_j(\w^\top a_j-c),0)$, $P=\frac12\|\cdot\|_2^2$, $g_1(\w,c)=\w-\overline{\w}$ and $g_2(\w,c)=-\w-\overline{\w}$. The strong duality holds since the constraints are linear, then the above problem is further equivalent to
\begin{equation*}
    \begin{array}{lll}
        \quad\max\limits_{\rho,\boldsymbol{\alpha}_1\geq0,\boldsymbol{\alpha}_2\geq0} & \min\limits_{\w,c,\z} & l(\w,c)+\lambda P(\z)+\bro^T(\w-\z)+\boldsymbol{\alpha}^T_1g_1(\w,c)+\boldsymbol{\alpha}^T_2g_2(\w,c) \\
       = \max\limits_{\bro,\boldsymbol{\alpha}_1\geq0,\boldsymbol{\alpha}_2\geq0} & -\max\limits_{\w,c,\z} & -\bro^\top\w-l(\w,c)+\rho^\top\z-\lambda P(\z)-(\boldsymbol{\alpha}^T_1g_1(\w,c)+\boldsymbol{\alpha}^T_2g_2(\w,c))\\
      =  \max\limits_{\bro,\boldsymbol{\alpha}_1\geq0,\boldsymbol{\alpha}_2\geq0} & & -l_g^*(-\bro,0)-\lambda P^*(\frac{\bro}{\lambda})\\
      =-\min\limits_{\bro,\boldsymbol{\alpha}_1\geq0,\boldsymbol{\alpha}_2\geq0} &&l_g^*(-\bro,0)+\lambda P^*(\frac{\bro}{\lambda}),
    \end{array}
\end{equation*}
where $l_g(\w,c)=l(\w,c)+\boldsymbol{\alpha}^T_1g_1(\w,c)+\boldsymbol{\alpha}^T_2g_2(\w,c)$. Then problem \eqref{eq:ref} for SVM is equivalent to
\begin{equation}\label{eq:llSVM}
    \begin{array}{ll}
         \min\limits_{\w,c} & L(\w,c) \\
         {\rm s.t.} & l(\w,c)+\lambda P(\w)\leq -l_g^*(-\bro)-\lambda P^*(\frac{\rho}{\lambda}).
    \end{array}
\end{equation}
Specifically, by Table \ref{table1}, we have
\begin{equation*}
    l_g(\w,c)=\sum\limits_{j\in I_{tr}} \max(1-b_j(\w^\top \a_j-c),0)+\boldsymbol{\alpha}^T_1(\w-\overline{\w})+\boldsymbol{\alpha}^T_2(-\w-\overline{\w}).
\end{equation*}
We calculate the conjugate function as follows
\begin{equation*}
    \begin{array}{lll}
        l_g^*(\y,t) & =\max\limits_{\w,c} & \{(\y^\top,t)(\w,c)^\top-\sum\limits_{j\in I_{tr}}\max(1-b_j(\w^\top a_j-c),0)-\boldsymbol{\alpha}^T_1(\w-\overline{\w})-\boldsymbol{\alpha}^T_2(-\w-\overline{\w})\}  \\
        & =\max\limits_{\w,c,\u} & \{(\y^\top,t)(\w,c)^\top-\sum\limits_{j\in I_{tr}}u_j-\boldsymbol{\alpha}^T_1(\w-\overline{\w})-\boldsymbol{\alpha}^T_2(-\w-\overline{\w})\}\\
        &\qquad \text{s.t.}&u_j\geq1-b_j(\w^\top a_j-c),u_j\geq0,j\in I_{tr}.\\
      \end{array}
\end{equation*}
Hence we have
\begin{equation}\label{eq:constrained}
\begin{array}{lll}
       l_g^*(\y,t) & =\max\limits_{\w,c,\u} & \{(\y^\top,t)(\w,c)^\top-\mathbf{1}^\top\u+(\boldsymbol{\alpha}^T_1+\boldsymbol{\alpha}^T_2)\overline{\w}-(\boldsymbol{\alpha}^T_1-\boldsymbol{\alpha}^T_2)\w\}\\
        & \qquad\text{s.t.}&u_j\geq1-b_j(\w^\top \a_j-c),z_j\geq0,j\in I_{tr}.\\
    \end{array}
\end{equation}
Note that \eqref{eq:constrained} is indeed a linear program and we simplify it by using duality.
Let $L_{\bga}$ denote the Lagrange function of \eqref{eq:constrained} and $\bga_1,\bga_2$ denote the multipliers. Then
\begin{equation}\label{eq:Lgamma}
    L_{\gamma}(\w,c,\u;\bga_1,\bga_2)=(\y,t)^T(\w,c)-\mathbf{1}^\top\z+(\boldsymbol{\alpha}_1+\boldsymbol{\alpha}_2)^T\overline{\w}-(\boldsymbol{\alpha}_1-\boldsymbol{\alpha}_2)^T\w+\bga_1^\top\u+\bga_2^\top(\u-\mathbf{1}+B_{tr}A_{tr}^\top\w-cB_{tr}\mathbf{1}).
\end{equation}
where $B_{tr}=\operatorname{diag}\{b_j,j\in I_{tr}\}$ is a diagonal matrix whose diagnal elements consist of $\{[b_{tr}]_i:i\in I_{tr}\}$ in sequence. By calculating the minimum value of the Lagrangian over $(\w,c,\u)$, we obtain the dual function as follows.

\begin{equation}\label{eq:llSVM2}
\begin{array}{lll}
       l_g^*(\y,t) & = \min\limits_{\y,t}&-\bga_2^\top\mathbf{1}+(\boldsymbol{\alpha}_1+\boldsymbol{\alpha}_2)^\top\overline{\w} \\
        & \qquad\text{s.t.}& \bga_1+\bga_2-\mathbf{1}=0,\\
        &&   \left(\begin{matrix}
        \y\\t
    \end{matrix}\right)+\left(\begin{matrix}
        A_{tr}^\top\\ \mathbf{1}^T
    \end{matrix}\right)B_{tr}\bga_2+\left(\begin{matrix}
         {\boldsymbol{\alpha}}_2- {\boldsymbol{\alpha}}_1\\0
    \end{matrix}\right)=0.
    \end{array}
\end{equation}
 By introducing $\r_2=\boldsymbol{\alpha}_1+\boldsymbol{\alpha}_2$ and recalling $\boldsymbol{\alpha}_1,\boldsymbol{\alpha}_2\geq0$, we conclude that \eqref{eq:llSVM2} is equivalent to
\begin{equation}\label{eq:llSVM2_0}
\begin{array}{lll}
       l_g^*(\y,t) & = \min\limits_{\y,t}&-\bga_2^\top\mathbf{1}+\r_2^\top\overline{\w} \\
        & \qquad\text{s.t.}& \bga_1+\bga_2-\mathbf{1}=0,\\
        && \boldsymbol{\alpha}_1+\boldsymbol{\alpha}_2=\r_2,\boldsymbol{\alpha}_1,\boldsymbol{\alpha}_2\geq0, \\
        &&   \left(\begin{matrix}
        \y\\t
    \end{matrix}\right)+\left(\begin{matrix}
        A_{tr}^\top\\ \mathbf{1}^T
    \end{matrix}\right)B_{tr}\bga_2+\left(\begin{matrix}
         {\boldsymbol{\alpha}}_2- {\boldsymbol{\alpha}}_1\\0
    \end{matrix}\right)=0.
    \end{array}
\end{equation}
Note that in \eqref{eq:llSVM}, $\lambda P^*(\frac{\bro}{\lambda})=\frac{\|\bro\|_2^2}{2\lambda}$. We introduce $\frac12\|\w\|_2^2\leq r_1,\frac{\|\bro\|_2^2}{2\lambda}\leq s$.
By combining \eqref{eq:llSVM} \eqref{eq:llSVM2} and using similar augments of Lemma \ref{pro:nonconvex}, we conclude that \eqref{eq:ref} is equivalent to \eqref{eq:svm2} for SVM. Using the expression $m_1(\lambda,r_1;\lambda^0,r_1^0) = \frac{1}{2}\left(\frac{\lambda^0}{r_1^0}r_1^2+\frac{r_1^0}{\lambda^0}\lambda^2\right)$ and $m_2(\overline{\w},\r_2;\overline{\w}^0,\r_2^0) = \frac{1}{2}\sum\limits_{j=1}^p\left(\frac{\overline{w}_j^0}{[r_2^0]_j}{[r_2]_j}^2+\frac{[r_2^0]_j}{\overline{w}_j^0}\overline{w}_j^2\right)$, \eqref{eq:svm1} follows.
\end{proof}
\subsection{Low-rank Matrix Completion}
We now discuss the low-rank matrix completion model and summarise problem as follows
\begin{equation}\label{eq:low_rank}
\begin{array}{cl}
     \min\limits_{\blam\in\R_+^{2G+1}} & \frac12\|M_{val}-X_{val}\bth\mathbf{1^T}-(Z_{val}\bbe\textbf{1}^T)^T-\Gamma\|_F^2 \\
     {\rm s.t.} & (\bbe,\Gamma)\in \argmin\limits_{\bbe,\Gamma}\frac12\|M_{tr}-X_{tr}\bth\mathbf{1^T}-(Z_{tr}\bbe\textbf{1}^T)^T-\Gamma\|_F^2\\
     & +\lambda_0\|\Gamma\|_*+\sum\limits_{g=1}^G\lambda_g\|\bth^{(g)}\|_2+\sum\limits_{g=1}^G\|\bbe^{(g)}\|_2
\end{array}
\end{equation}
where $M_{val}=\{M_{ij}\}_{(i,j)\in\Omega_{val}}, M_{tr}=\{M_{ij}\}_{(i,j)\in\Omega_{tr}}$, $X_{val}=(\x_i)_{i\in I_{val}},X_{tr}=(\x_i)_{i\in I_{tr}}$ and $Z_{val}=(\z_j)_{j\in I_{val}},Z_{tr}=(\z_j)_{j\in I_{tr}}$
\begin{proposition}
We denote the spectral norm of $W$ as $\|W\|_p$ and corresponding matrice as above. \eqref{eq:refsub} for problem \eqref{eq:low_rank} can be reformulated as
\begin{equation}\label{low-rank-convex}
    \begin{array}{cl}
        \min\limits_{\blam\in\R_+^{2G+1}} & \frac12\|M_{val}-X_{val}\bth\mathbf{1^T}-(Z_{val}\bbe\textbf{1}^T)^T-\Gamma\|_F^2 \\
        {\rm s.t.} & \frac12\|M_{tr}-X_{tr}\bth\mathbf{1^T}-(Z_{tr}\bbe\textbf{1}^T)^T-\Gamma\|_F^2+\operatorname{tr}(M_{tr}^TW)+\frac12\|W\|^2_F+\frac12\sum\limits_{g=0}^{2G}\frac{\lambda_g^0}{r_g^0}r_g^2+\frac{r_g^0}{\lambda_g^0}\lambda_g^2\leq0, \\
        &-\boldsymbol{\rho}_1+X_{tr}^TW\mathbf{1}=\bz,\\
        &-\boldsymbol{\rho}_2+Z_{tr}^TW^T\mathbf{1}=\bz,\\
        &\| {\Gamma}\|_*\leq r_0,\\
        &\left\|\bth^{(g)}\right\|_2\leq r_g,\ g=1,...,G,\\
        &\left\|\bbe^{(g)}\right\|_2\leq r_{g+G},\ g=1,...,G,\\
        &\left\|\boldsymbol{\rho}_1^{(g)}\right\|_2\leq \lambda_g,\ g=1,...,G,\\
        &\left\|\boldsymbol{\rho}_2^{(g)}\right\|_2\leq \lambda_{g+G},\ g=1,...,G,\\
        &\|W\|_p\leq \lambda_0.\\
    \end{array}
\end{equation}
Moreover, \eqref{eq:ref} for low-rank matrix completion is
\begin{equation}\label{low-rank-bilinear}
    \begin{array}{cl}
        \min\limits_{\blam\in\R_+^{2G+1}} & \frac12\|M_{val}-X_{val}\bth\mathbf{1^T}-(Z_{val}\bbe\textbf{1}^T)^T-\Gamma\|_F^2 \\
        {\rm s.t.} & \frac12\|M_{tr}-X_{tr}\bth\mathbf{1^T}-(Z_{tr}\bbe\textbf{1}^T)^T-\Gamma\|_F^2+\operatorname{tr}(M_{tr}^TW)+\frac12\|W\|^2_F+\sum\limits_{g=0}^{2G}\lambda_gr_g\leq0, \\
        &-\boldsymbol{\rho}_1+X_{tr}^TW\mathbf{1}=\bz\\
        &-\boldsymbol{\rho}_2+Z_{tr}^TW^T\mathbf{1}=\bz\\
        &\| {\Gamma}\|_*\leq r_0,\\
        &\left\|\bth^{(g)}\right\|_2\leq r_g,\ g=1,...,G,\\
        &\left\|\bbe^{(g)}\right\|_2\leq r_{g+G},\ g=1,...,G,\\
        &\left\|\boldsymbol{\rho}_1^{(g)}\right\|_2\leq \lambda_g,\ g=1,...,G,\\
        &\left\|\boldsymbol{\rho}_2^{(g)}\right\|_2\leq \lambda_{g+G},\ g=1,...,G,\\
        &\|W\|_p\leq \lambda_0.\\
    \end{array}
\end{equation}
    
\end{proposition}
\begin{proof}
    We define
    \begin{equation*}
        l(\bth,\bbe,\Gamma)=\frac12\|M_{tr}-X_{tr}\bth\mathbf{1^T}-(Z_{tr}\bbe\textbf{1}^T)^T-\Gamma\|_F^2,
    \end{equation*}
    and compute its conjugate function
    \begin{equation*}
        l^*(\mathbf{u},\mathbf{v},W)=\max\limits_{\bth,\bbe,\Gamma}g(\bth,\bbe,\Gamma):=\max\limits\bth^T\mathbf{u}+\bbe^T\mathbf{v}+\operatorname{tr}(\Gamma^TW)-l(\bth,\bbe,\Gamma).
    \end{equation*}
    By first order condition, for optimal $(\bth^*,\bbe^*,\Gamma^*)$, it holds that
    \begin{eqnarray*}
        \nabla_\bth g(\bth,\bbe,\Gamma)& = & 0,\\
        \nabla_\bbe g(\bth,\bbe,\Gamma)& = & 0,\\
        \nabla_\Gamma g(\bth,\bbe,\Gamma)& = & 0,
    \end{eqnarray*}
    which are equivalent to
    \begin{eqnarray}
        \mathbf{u}+X_{tr}^T(M_{tr}-X_{tr}\bth\mathbf{1}^T-(Z_{tr}\bbe\mathbf{1}^T)^T-\Gamma)\mathbf{1}& = & \mathbf{0},\label{gra_bth}\\
        \mathbf{v}+Z_{tr}^T(M_{tr}-X_{tr}\bth\mathbf{1}^T-(Z_{tr}\bbe\mathbf{1}^T)^T-\Gamma)^T\mathbf{1}&=&\bz,\label{gra_bbe}\\
        W-(M_{tr}-X_{tr}\bth\mathbf{1}^T-(Z_{tr}\bbe\mathbf{1}^T)^T-\Gamma)=\bz.\label{gra_gamma}
    \end{eqnarray}
    Substituting \eqref{gra_gamma} into \eqref{gra_bbe} and \eqref{gra_bth}, we obtain
    \begin{equation}\label{eq:W}
    \u-X_{tr}^TW\mathbf{1}=\bz,\quad \v-Z_{tr}^TW^T\mathbf{1}=\bz.
    \end{equation}
    Hence $l^*(\u,\v,W)=+\infty$ if \eqref{eq:W} is not satisfied.  We choose $\bth^*=\bz,\bbe^*=\bz,\Gamma^*=M_{tr}+W$ and obtain that
    \begin{equation}\label{eq:low_rank_con}
        l^*(\u,\v,W)=\begin{cases}
    \operatorname{tr}(M_{tr}^TW)+\frac12\|W\|_F^2\text{ if } \eqref{eq:W} \text{ holds},\\
    +\infty\text{  otherwise.}
    \end{cases}
    \end{equation}
    By combining \eqref{eq:low_rank_con} and using similar augments of Lemma \ref{pro:nonconvex}, we conclude that \eqref{eq:low_rank} is equivalent to for low-rank matrix completion. Using the expression $m(\lambda_g,r_g;\lambda_g^0,r_g^0)=\frac12\left(\frac{\lambda_g^0}{r_g^0}r_g^2+\frac{r_g^0}{\lambda_g^0}\lambda_g^2\right)$, \eqref{low-rank-convex} follows.

\end{proof}

\section{Detials for Experiments and Data}\label{detail-data}
\subsection{Elastic Net}\label{detail_sy_ela}
The generation of feature matrix $A\in\R^{n\times p}$ and response vector $\mathbf{b}\in\R^n$ follows \cite{feng2018gradient}, where the column vectors ${\a_i\in\R^p:i\in I_{tr}\setminus I_{val}}$ satisfy the marginal distribution $N(\mathbf{0},\mathbf{I})$, and the correlation matrix between column vectors satisfies $cor(a_{ij},a_{ik})=0.5^{|j-k|}$. The feature matrix is full rank and satisfies the conditions in Lemma \ref{co:ls}. Next, we generate a random vector $\boldsymbol{\beta}\in\R^p$ with 15 non-zero elements, where each element $\beta_i$ is either 0 or 1. The response vector $\mathbf{b}$ is obtained by applying the feature matrix to the random vector and adding a certain amount of noise, i.e., $\mathbf{b}=\mathbf{A}\boldsymbol{\beta}+\sigma\boldsymbol{\epsilon}$, where we set the signal-to-noise ratio to $\sigma=2$ and the noise $\boldsymbol{\epsilon}\sim N(\mathbf{0},\mathbf{I}_n)$. Random search is implemented using 100 uniformly random samples. The variable space of TPE is set to a uniform distribution on $[-5,2]$ for both $u_1$ and $u_2$. We follow \cite{gao2022value} and use the same parameter settings and stopping criteria to implement the VF-iDCA algorithm. For LDMMA algorithm, we set the initial point to $\boldsymbol{\lambda}^0=(0.01,0.01)$. For the $\epsilon$-perturbation problem \eqref{eq:blpeps}, we set $\epsilon=0.01$.

\subsection{Sparse Group Lasso}
The generation of the feature matrix $A\in\R^{n\times p}$ and the response vector $\mathbf{b}\in\R^n$ follows \cite{feng2018gradient}. The generated dataset includes $n$ training samples, $n/3$ validation samples, and 100 fixed testing samples. The observation matrix $\mathbf{A}$ satisfies that each column vector $\a_i$ follows the standard normal distribution. The random vector $\boldsymbol{\beta}=[\boldsymbol{\beta}^{(1)},\boldsymbol{\beta}^{(2)},\boldsymbol{\beta}^{(3)}]\in\R^p$, where $\boldsymbol{\beta}^{(i)}=(1,2,3,4,5,0,...,0)$. The response vector $\mathbf{b}$ is generated by applying the feature matrix to the random vector and adding some noise. Specifically, $b_i=\boldsymbol{\beta}^T\a_i+\sigma\epsilon_i$ and $\mathbf{b}=(b_1,b_2,...,b_n)$. Like the elastic net model, we set the signal-to-noise ratio to $\sigma=2$ and the noise $\boldsymbol{\epsilon}=(\epsilon_1,...,\epsilon_n)\sim N(\mathbf{0},\mathbf{I})$. For the experiments with four different data sizes, the algorithm details of VF-iDCA are the same as \cite{gao2022value}, including parameter tuning. For LDMMA, the initial value of the iteration is set to $\boldsymbol{\lambda}^0=(0.1,...,0.1)$, and $\epsilon=1$ is also set. It is worth noting that if $\epsilon$ is too small, the feasible domain of problem \eqref{eq:blpeps} is insufficient to complete the iteration. This premature termination will result in an abnormally low validation error and a larger test error. In machine learning, we call it overfitting, which is usually caused by poor generalization performance of the model. The occurrence of overfitting indicates that the model only has a good learning effect on the training and validation data but has no practical value for unlearned test data and more extensive data. Overall, in this series of experiments, we need to choose an appropriate value of $\epsilon$, which can avoid overfitting and prevent the solution of problem \eqref{eq:blpeps} from deviating too much from the solution of the original problem \eqref{eq1}.



\subsection{Support Vector Machine}
We fetch the datasets with libSVM toolbox and obtain the corresponding observation matrix and label vector of all datasets. Each dataset is divided into two seperate parts: a cross-validation training set $\Omega$ containing $3\lfloor N/6\rfloor$ samples, and a test set $\Omega_{test}$ containing the remaining samples. Based on this devision, we partition the entire training set into multiple equal parts and iteratively use one part as the validation set and the remaining parts as the training set to solve the SVM problem. In the experiment, we performed 3-fold and 6-fold cross-validation on the training and validation sets for each of the six datasets to optimize hyperparameters.

Finally, we use the obtained hyperparameters and corresponding model to compute the error on the validation set. We repeat this process for each part to reduce the impact of data variability on the model.  However, in the process of solving the SVM problem, cross-validation is involved, and the obtained hyperparameters satisfy the minimization of the lower-level function of the SVM problem. Therefore, we need to use the hyperparameters to solve the upper-level problem to obtain the corresponding validation and test errors. We randomly divide the cross-validation training set $\Omega$ into $K$ mutually exclusive subsets $\{\Omega_{val}^k\}_{k=1}^K$, each of which will be used as the validation set. The remaining parts will be used as the training set $\Omega_{tr}^k=\Omega\setminus\Omega_{val}^k$. We define the loss function on the validation set in the cross-validation process as:
\begin{equation}\label{eq:Theta_val}
\begin{array}{l}
      \Theta_{val}(\w^1,\w^2,\dots,\w^K,\mathbf{c})   := \frac{1}{K}\sum\limits_{k=1}^K\frac{1}{|\Omega_{val}^k|}\sum\limits_{j\in \Omega_{val}^k}\max(1-b_j(\a_j^\top\w^k-c^k),0),
\end{array}
\end{equation}
The primal problem of the support vector machine \eqref{eq:SVMprimal} is then transformed into the following bilevel program \citep{kunapuli2008classification}:
\begin{equation}\label{eq:SVMbilevel}
    \begin{array}{ll}
         \min\limits_{\w,c} & \Theta_{val}(\w^1,\w^2,\dots,\w^K,\mathbf{c}) \\
         {\rm s.t.} & \lambda>0,\bar{\w}_{lb}\leq\bar{\w}\leq\bar{\w}_{ub}\\
         & (\w^k,c^k)\in\mathop{\arg\min}\limits_{-\overline{\w}\leq\w\leq\overline{\w}}\left\{\sum\limits_{j\in \Omega_{tr}^k}\max(1-b_j(\a_j^\top\w-c),0)+\frac{\lambda}{2}\|\w\|_2^2\right\},\\
         & k=1,2,\dots,K.
    \end{array}
\end{equation}
where $\mathbf{c}=(c^1,c^2,\dots,c^K)$, $c^1,c^2,\dots,c^K$ and $\w^1,\w^2,\dots,\w^K$ are $K$ parallel copies of $c$ and $\w$. $\bar{\w}{lb}$ and $\bar{\w}{lb}$ are the upper and lower bounds of $\bar{\w}$, respectively. We can define a loss function on the test set analogous to \eqref{eq:Theta_val}:
\begin{equation}\label{eq:Theta_tr}
\begin{array}{ll}
     & \Theta_{tr}(\w^1,\w^2,\dots,\w^K,\mathbf{c})   \\
     & := \frac{1}{K}\sum\limits_{k=1}^K\frac{1}{|\Omega_{tr}^k|}\sum\limits_{j\in \Omega_{tr}^k}\max(1-b_j(\a_j^\top\w^k-c^k),0),
\end{array}
\end{equation}
Correspondingly, the subproblem \eqref{eq:svm2} to be solved is transformed into:
\begin{equation}\label{eq:svm_fold}
    \begin{array}{lcl}
        &\min\limits_{\lambda,\bar{\w},\w^1,\w^2,\dots,\w^K,c} & \Theta_{val}(\w^1,\w^2,\dots,\w^K,\mathbf{c}) ,\\
        & {\rm s.t.} & \Theta_{tr}(\w^1,\w^2,\dots,\w^K,\mathbf{c})+r_1\lambda+\mathbf{r}_2^\top\overline{\w}+s-\mathbf{1}^T \mathbf{v}\leq0,\\
      &&    \frac{1}{2}\|\w^k\|^2\leq r_1,k=1,2,\dots,K\\
        &&  \left\|\begin{matrix}
             \sqrt{2}\bro\\ \lambda-s
         \end{matrix}\right\|_2\leq \lambda+s,\\
    &&      \left(\begin{matrix}
        {A_{tr}^k}^T\\ \mathbf{1}^T
        \end{matrix}\right)B_{tr}^k\mathbf{v}+\left(\begin{matrix}
        -\bro\\
        0
        \end{matrix}\right)+\left(\begin{matrix}
         {\boldsymbol{\alpha}}_2- {\boldsymbol{\alpha}}_1\\0
    \end{matrix}\right)=0,k=1,2,...,K\\
    && \boldsymbol{\alpha}_1+\boldsymbol{\alpha}_2=\mathbf{r}_2,\\
   &&      0\leq \mathbf{v}\leq1,\boldsymbol{\alpha}_1,\boldsymbol{\alpha}_2\geq0,\\
         &&-\overline{\w}\leq\w^k\leq\overline{\w},k=1,2,\dots,K\\
         &&\bar{\w}_{lb}\leq\bar{\w}\leq\bar{\w}_{ub}
    \end{array}
\end{equation}
Finally, we substitute the optimal solutions $\bar{\w},\lambda$ obtained from the above problem into the following problem and solve it again to obtain the optimal $(\w,c)$,
\begin{equation*}
    (\w,c)\in\mathop{\arg\min}\limits_{-\overline{\w}\leq\w\leq\overline{\w}}\left\{\sum\limits_{j\in \Omega}\max(1-b_j(\a_j^\top\w-c),0)+\frac{\lambda}{2}\|\w\|_2^2\right\}.
\end{equation*}

We use MOSEK solver to handle the 2-norm term in the objective function of the upper-level problem, which is convex and smooth. We also conduct VF-iDCA and other methods according to the setting in \cite{gao2022value}. For LDMMA, we set the initial point of the iteration to $\boldsymbol{\lambda}^0=(0.1,\dots,0.1)$ and parameters $\bar{\w}_{lb}=(10^{-6},\dots,10^{-6})$, $\bar{\w}_{ub}=(10,\dots,10)$ for the lower and upper bounds of $\bar{\w}$, respectively. We choose $\epsilon=1$ for 3-fold cross-validation and $\epsilon=5$ for 6-fold cross-validation to ensure the primal and dual feasibility of the subproblems in each iteration, which is crucial for MOSEK and prevents overfitting.

\subsection{Elastic Net with High Dimensional Datasets}
Compared with Section \ref{detail_sy_ela}, we only replace the synthetic datasets with real datasets. The gisette dataset comprises 5000 features and 6000 samples, whereas the sensit dataset encompasses 78823 features. For dataset partition, we extract 50, 25 examples as training set and 50, 25 examples as validation set, respectively. We set the initial point as $\boldsymbol{\lambda}^0=(0.01,0.01)$ and perturbation parameter $\epsilon=1$ for our algorithm. Meanwhile, we also conduct VF-iDCA and other methods according to the setting in \cite{gao2022value}.

\begin{table}[H]
\tiny
\caption{Support Vector Machine problems with 3-fold and 6-fold cross-validation on three datasets, where the number of features $p$ and samples $|\Omega|,|\Omega_{test}|$ are displayed together with dataset names.}
\label{table-SVM-3fold-6fold-app}
\centering
\small
\setlength\tabcolsep{4pt}
\resizebox{\linewidth}{!}{
\begin{tabular}{ll|ccc|ccc}
\toprule
 \multirow{2}{*}{\textbf{Dataset}} & \multirow{2}{*}{\textbf{Methods}} & \multicolumn{3}{c}{\textbf{3-fold}} & \multicolumn{3}{c}{\textbf{6-fold}} \vspace{2pt}\\
 \cline{3-5}\cline{6-8} \multicolumn{1}{c}{}& \multicolumn{1}{c}{} & \textbf{Times(s)} & \textbf{Val. Err.} & \textbf{Test Err.} & \textbf{Times(s)} & \textbf{Val. Err.} & \textbf{Test Err.}\\
 \midrule
 \multirow{5}{*}{\makecell[l]{liver-disorders-scale\\ $p=5$ \\ $|\Omega|=72$ \\ $|\Omega_{test}|=73$}} & Grid & $0.74\pm0.01$ & $0.65\pm0.08$ & $0.32\pm0.07$ & $1.16\pm0.02$ & $0.61\pm0.08$ & $0.32\pm0.06$\\
 & Random & $0.75\pm0.02$ & $0.63\pm0.07$ & $0.32\pm0.05$ & $1.16\pm0.04$ & $0.59\pm0.06$ & $0.32\pm0.05$ \\
 & TPE & $0.68\pm0.55$ & $0.65\pm0.08$ & $0.32\pm0.07$ & $2.26\pm1.67$ & $0.62\pm0.06$ & $0.32\pm0.06$ \\
 & VF-iDCA & $0.13\pm0.03$ & $0.52\pm0.07$ & $0.27\pm0.04$ & $0.27\pm0.03$ & $0.40\pm0.05$ & $0.30\pm0.04$ \\
 & LDMMA & $0.08\pm0.01$ & $0.46\pm0.08$ & $0.23\pm0.10$ & $0.15\pm0.04$ & $0.19\pm0.08$ & $0.24\pm0.08$
 \vspace{2pt}\\
 \hline \\ [-1.8ex]
 \multirow{5}{*}{\makecell[l]{diabetes-scale \\ $p=8$ \\ $|\Omega|=384$ \\ $|\Omega_{test}|=384$}} & Grid & $3.17\pm0.08$ & $0.55\pm0.03$ & $0.19\pm0.03$ & $6.22\pm0.21$ & $0.54\pm0.03$ & $0.33\pm0.04$\\
 & Random & $3.47\pm0.14$ & $0.56\pm0.03$ & $0.32\pm0.05$ & $7.18\pm0.30$ & $0.55\pm0.04$ & $0.30\pm0.05$ \\
 & TPE & $10.21\pm6.68$ & $0.55\pm0.04$ & $0.29\pm0.06$ & $76.67\pm36.39$ & $0.54\pm0.03$ & $0.34\pm0.06$ \\
 & VF-iDCA & $0.28\pm0.04$ & $0.48\pm0.03$ & $0.23\pm0.01$ & $0.65\pm0.03$ & $0.43\pm0.03$ & $0.23\pm0.02$ \\
 & LDMMA & $0.22\pm0.03$ & $0.49\pm0.02$ & $0.19\pm0.01$ & $0.55\pm0.10$ & $0.39\pm0.05$ & $0.20\pm0.02$
 \vspace{2pt}\\
 \hline \\ [-1.8ex]
 \multirow{5}{*}{\makecell[l]{breast-cancer-scale \\ $p=14$ \\ $|\Omega|=336$ \\ $|\Omega_{test}|=347$}} & Grid & $3.32\pm0.09$ & $0.08\pm0.01$ & $0.16\pm0.08$ & $6.32\pm0.11$ & $0.08\pm0.01$ & $0.15\pm0.12$\\
 & Random & $3.69\pm0.07$ & $0.09\pm0.01$ & $0.08\pm0.08$ & $7.20\pm0.12$ & $0.09\pm0.02$ & $0.10\pm0.11$ \\
 & TPE & $17.88\pm10.05$ & $0.09\pm0.01$ & $0.10\pm0.11$ & $34.66\pm20.57$ & $0.09\pm0.01$ & $0.18\pm0.13$ \\
 & VF-iDCA & $0.24\pm0.04$ & $0.09\pm0.01$ & $0.04\pm0.01$ & $0.57\pm0.12$ & $0.08\pm0.01$ & $0.03\pm0.01$ \\
 & LDMMA & $0.12\pm0.01$ & $0.08\pm0.01$ & $0.03\pm0.01$ & $0.42\pm0.17$ & $0.08\pm0.01$ & $0.02\pm0.01$
 \vspace{2pt}\\
 \hline \\ [-1.8ex]
 \multirow{5}{*}{\makecell[l]{ sonar \\ $p=60$ \\ $|\Omega|=102$ \\ $|\Omega_{test}|=106$}} & Grid & $10.08\pm0.33$ & $0.59\pm0.10$ & $0.41\pm0.14$ & $20.88\pm0.61$ & $0.63\pm0.06$ & $0.49\pm0.12$\\
 & Random & $10.30\pm0.18$ & $0.55\pm0.07$ & $0.31\pm0.08$ & $20.56\pm0.31$ & $0.58\pm0.03$ & $0.41\pm0.10$ \\
 & TPE & $42.80\pm13.95$ & $0.64\pm0.13$ & $0.45\pm0.11$ & $189.82\pm19.80$ & $0.70\pm0.06$ & $0.53\pm0.07$\\
 & VF-iDCA & $1.32\pm0.23$ & $0.03\pm0.02$ & $0.25\pm0.04$ & $3.03\pm0.09$ & $0.00\pm0.00$ & $0.24\pm0.04$\\
 & LDMMA & $0.82\pm0.15$ & $0.17\pm0.02$ & $0.25\pm0.04$ & $2.38\pm0.19$ & $0.00\pm0.00$ & $0.22\pm0.02$
 \vspace{2pt}\\
 \hline \\ [-1.8ex]
 \multirow{5}{*}{\makecell[l]{a1a \\ $p=123$ \\ $|\Omega|=801$ \\ $|\Omega_{test}|=804$}} & Grid & $17.07\pm0.36$ & $0.41\pm0.02$ & $0.24\pm0.02$ & $36.77\pm0.99$ & $0.39\pm0.02$ & $0.24\pm0.01$ \\
 & Random & $17.81\pm0.30$ & $0.41\pm0.02$ & $0.21\pm0.03$ & $39.03\pm0.65$ & $0.39\pm0.02$ & $0.21\pm0.02$\\
 & TPE & $187.91\pm39.92$ & $0.42\pm0.02$ & $0.23\pm0.02$ & $447.17\pm85.49$ & $0.40\pm0.02$ & $0.24\pm0.01$\\
 & VF-iDCA & $2.40\pm0.13$ & $0.27\pm0.02$ & $0.17\pm0.01$ & $11.01\pm1.26$ & $0.19\pm0.02$ & $0.18\pm0.01$\\
 & LDMMA & $1.24\pm0.12$ & $0.20\pm0.02$ & $0.15\pm0.08$ & $8.04\pm0.71$ & $0.15\pm0.05$ & $0.17\pm0.01$
 \vspace{2pt}\\
 \hline \\ [-1.8ex]
 \multirow{5}{*}{\makecell[l]{w1a \\ $p=300$ \\ $|\Omega|=1236$ \\ $|\Omega_{test}|=1241$}} & Grid & $20.08\pm0.33$ & $0.59\pm0.10$ & $0.41\pm0.14$ & $104.47\pm2.99$ & $0.06\pm0.01$ & $0.03\pm0.00$\\
 & Random & $20.30\pm0.18$ & $0.55\pm0.07$ & $0.31\pm0.08$ & $147.88\pm8.64$ & $0.05\pm0.00$ & $0.02\pm0.00$\\
 & TPE & $85.80\pm13.95$ & $0.64\pm0.13$ & $0.45\pm0.11$ & $682.35\pm17.52$ & $0.06\pm0.01$ & $0.03\pm0.00$\\
 & VF-iDCA & $4.32\pm0.23$ & $0.03\pm0.02$ & $0.03\pm0.00$ & $25.37\pm3.10$ & $0.01\pm0.00$ & $0.03\pm0.00$\\
 & LDMMA & $2.19\pm0.24$ & $0.01\pm0.00$ & $0.01\pm0.00$ & $15.25\pm2.90$ & $0.01\pm0.00$ & $0.02\pm0.00$\\
\bottomrule
\end{tabular}
}
\end{table}

\end{document}